\documentclass[12pt]{amsart}

\usepackage{Tdefn}

\title{Lattice Spectral Sequences and  Cohomology of Configuration Spaces }
\date{\today}
\author{Philip Tosteson}

\begin{document}
\maketitle

\begin{abstract}
   	For a topological space $X$, we introduce a criterion for the $\FI$ module $n \mapsto H^i(\Conf_n(X))$ to be finitely generated and give several applications.  For instance,  if $C$ is a finite connected $CW$ complex, then $X = C \times \R^2$ satisfies the criterion.   Our main  tool is a spectral sequence that we derive from the homological algebra of the partition lattice.  
\end{abstract}
	\section{Introduction}
		\subsection{Cohomology of configuration spaces}\label{confintro}
		
		Let $X$ be a topological space and $k$ be a field.  In this paper, we study the behavior of $H^\bdot (\Conf_n(X), k)$,  the cohomology of configurations of $n$ ordered points in  $X$, for $n \gg 0$.  We use the theory of $\FI$  modules, following the approach developed by Church, Ellenberg, and Farb \cite{CEF}.  $\FI$ is the category of finite sets and injections, and an $\FI$ module is a functor from $\FI$ to the category of $k$ modules.    
		An injection $[m] \into [n]$ yields a map $\Conf_m(X) \leftarrow \Conf_n(X)$ by forgetting and relabeling the points.   This gives the cohomology of configuration space the structure of an  $\FI$ module,  $n \mapsto 
		H^i(\Conf_n (X), k)$.  
		
		When $M$ is an orientable manifold of dimension $\geq 2$, Church, Ellenberg, and Farb prove that  $H^i(\Conf_\bdot (M), \Q)$ is a finitely generated $\FI$ module.   The main purpose of this paper is to extend the finite generation of $H^i(\Conf_\bdot(X),k)$ to topological spaces $X$ that are in some sense $\geq 2$ dimensional.

 	It is important to show that $H^i(\Conf_\bdot (X),k)$ is finitely generated  because  finitely generated $\FI$ modules exhibit uniform behavior for large $n$.  As shown in \cite{CEF}, finite generation of $H^i(\Conf_\bdot(X), \Q)$  implies that the character of $H^i(\Conf_n(X), \Q)$ as an $\bS_n$ representation agrees with a character polynomial for $n \gg 0$. Thus finite generation is closely related to the representation stability results of \cite{church2012homological, church2013representation}.   Taking the multiplicity of the trivial representation recovers the classical homological stability of unordered configuration space \cite{ mcduff1975configuration, segal1979topology}. In positive characteristic, Nagpal shows that finite generation of $H^i(\Conf_\bdot (X), \bbF_p)$ as an $\FI$ module implies periodicity in the cohomology of unordered configuration space \cite{nagpal2015fi}.


		First, we give an explicit criterion for the sheaf cohomology $H^i( \Conf_\bdot (X), k)$ to be a finitely generated $\FI$ module in terms of the vanishing of the relative cohomology of $(X^n , X^n - \Delta X)$.

		\begin{thm}[Criterion for finite generation] \label{introcriterion}
			Let $k$ be a field and let $X$ be a Hausdorff, connected, locally contractible topological space with $H^i(X, k)$  finite dimensional for $i \geq 0$. Let $c \in \bbN$. Write $\Delta: X \to X^n$ for the diagonal embedding.    Suppose the following conditions hold:
			 \begin{enumerate}
				\item   For $n \geq 2$ and  $i < n$,    we have $H^i(X^n, X^n - \Delta X,k) = 0$, 
				\item  For $n \gg 0$, we have $H^i(X^n, X^n - \Delta X, k) = 0$ for $i \leq n + c$, 
			\end{enumerate}
			Then $H^i(\Conf_n(X), k)$ is a finitely generated {\rm $\FI$} module for all $i < c$.
		\end{thm}

		 Notice for $X = \bbR^d$ we have that $H^\bdot (\bbR^{dn}, \bbR^{dn} - \Delta \R^d, k)$ equals $k$ in degree $(d-1)n$, thus $\bbR^d$ satisfies conditions $(1)$ and $(2)$  exactly when $d \geq 2$.  When $d = 1$, we have $H^0(\Conf_n \R^1, k) =k[\bS_n]$, which grows too quickly to be a finitely generated $\FI$ module. 
		 
		    Throughout this paper, we take $H^i( Y, k)$ to mean the sheaf cohomology of $Y$, which agrees with singular cohomology whenever $Y$ is locally contractible \cite{sella2016comparison}.  When $X$ is Hausdorff and locally contractible, $Y = \Conf_n(X)$ is also locally contractible. 
		
		Interpreting $H^i(X^n, X^n - \Delta X, k)$ as the hypercohomology of a complex of sheaves gives a corollary of Theorem \ref{introcriterion} that is local on $X$.   If $X$ has an open cover by spaces of the form $V \times \R^2$  then  the Kunneth formula and the vanishing of $H^i(\bbR^{2n}, \bbR^{2n} - \bbR^2, k)$ together imply that $X$ satisfies this local criterion:  
		
		\begin{cor}\label{timesr2}
  			Let $X$ be a Hausdorff, connected, locally contractible topological space with $H^i(X, k)$ finite dimensional.  If every point $p$ has a neighborhood $U_p \iso V_p \times \R^2$ for some space $V_p$,  then $H^i(\Conf_\bdot(X) , k)$ is a finitely generated{ \rm $\FI$} module for all $i$.  
  		\end{cor}

		The following examples satisfy the hypotheses of Corollary \ref{timesr2}.  We believe they are new. 
		\begin{exa}
  			Let $W$ be connected topological space  with finite dimensional cohomology.  Suppose that $W$ is either a topologically stratified  space\cite{thom1969ensembles}, or a Whitney stratified spacce  \cite{goresky1988stratified}.   If all the strata  of $W$ are $\geq 2$ dimensional, then $H^i(\Conf_\bdot(W) , k)$ is a finitely generated  $\FI$ module.  
  		\end{exa}

		\begin{exa}
			Let  $Y$ be a connected $CW$ complex with finitely many cells.  Then $H^i(\Conf_n(Y \times \bbR^2) , k)$ is a finitely generated $\FI$ module. 
		\end{exa}

		The relative cohomology of $(X^n, X^n - \Delta X)$  can be expressed in terms of the dual of tensor powers  of $\omega_X$, where $\omega_X $ is the dualizing complex of $X$, whenever the hypotheses of Verdier duality are satisfied \cite{kashiwara2013sheaves}.  Under this reformulation, if $\omega_X$ is concentrated in homological degree $\geq 2$,  then the two vanishing conditions of Theorem \ref{introcriterion} follow.  Further, the homology of the stalk of $\omega_X$ at $p$ is $\tilde H_{i-1}(U-p, k)$, where $U$ is a contractible neighborhood of $p$.  From this we obtain a criterion for finite generation that depends on whether, locally, removing a point disconnects $X$.   

  		\begin{cor} \label{dualcorr2}
  				Let $k$ be a field and let $X$ be a connected, locally contractible,  closed subset of $\bbR^n$  such that  $H^i(X, k)$ is finite dimensional and $H^r_c(X,k)$ vanishes for $r \gg 0$.  
  				If for all $p \in X$ there is a contractible neighborhood $U$ of $p$ such that $U-p$ is connected, then $H^i(\Conf_\bdot (X))$ is a finitely generated  {\rm $\FI$} module for all $i$.  
  
  		\end{cor}
  		
  			Notice that  $\geq 2$ dimensional manifolds satisfy these hypotheses, while $1$ dimensional manifolds do not. The following examples are applications of Corollary \ref{dualcorr2},

  		  \begin{ex}
  		  	Let $G$ be a connected finite graph, considered as a CW complex.   Then  $H^i(\Conf_\bdot(G \times \R^1), k)$ is a finitely generated $\FI$ module for all $i$.     
		\end{ex}
  		
  		\begin{ex}
  			Let $X$ consist of two solid balls glued together along a closed interval on their boundary running from the north to south pole,   $X = D^3 \cup_{[0,1]} D^3$.  Then $H^i(\Conf_{\bdot}(X), k)$ is a finitely generated $\FI$ module for all $i$.  
  		\end{ex}
  		
 	
		  \medskip
	
	\subsection{Spectral sequence}\label{posethomologicalalgebra}

		In \cite{CEF}, the authors pass from local to global via the Leray spectral sequence for $\Conf_n(M) \to M^n$,  studied by Totaro in  \cite{totaro1996configuration}.  Totaro uses the cohomology of $\Conf_n(\R^d)$ to give a description of the $E_2$ page of this spectral sequence.   When $X$ is a general topological space we do not know the local structure of its configuration space, and so a new tool is required.  
		
		Our main innovation is to replace the Totaro spectral sequence by a similar one, derived from the homological algebra of the partition lattice.   Theorem \ref{introcriterion} then follows from analyzing the $E_1$ page.  The second purpose of this paper is to construct this spectral sequence in a general context.

		Let $Y$ be a topological space, $k$ a commutative ring, and let $\cF$ be a sheaf of $k$  modules on $Y$.  Let $P$ be a finite meet lattice.  Our convention is that ${ P}$ does not have a top element, and that ${\widehat P}$  denotes $P \cup \hat 1$, where $\hat 1$ is a top element.  Let $\{Z_m\}_{m \in {\widehat P}}$ be closed subsets of $Y$ indexed by $ \hat P$ such that $Z_m\cap Z_l = Z_{m \wedge l}$,  $Z_{\hat 1} = Y$ and $Z_m \subset Z_l$ for $m \leq l$.  
		 We give a spectral sequence that computes $H^i (Y - \bigcup_{m \in P} Z_m , \cF)$ from the relative cohomology groups  $\{H^i(Y, Y - Z_m, \cF)\}_ {m \in {\widehat P}}$, and homology groups that depend on ${\widehat P}$.  
		
		\begin{thm}\label{vagueFISS}
			    We construct a spectral sequence converging to $H^\bdot (Y - \bigcup_{m \in P} Z_m , \cF)$.  The $E_1$ page consists of sums of $k$ modules of the form: $$\Tor^{\widehat P}_j(S_{\hat 1},  S_m \otimes_k H^i( Y, Y - Z_m, \cF ) ), ~m\in {\widehat P}, ~i,j \in \bbN$$ in cohomological degree $i-j$.  This spectral sequence depends on the choice of a function  $r : {\widehat P} \to \N$ such that $m < l \in {\widehat P}$ implies $r(m) > r(l)$.   When  $m \neq \hat 1$, $$\Tor^{\widehat P}_j(S_{\hat 1},  S_m \otimes_k H^i( Y, Y - Z_m, \cF ) ) =  \tilde H_{j-2}( (m, \hat 1),  H^i(Y, Y- Z_m)),$$  where $ \tilde H_{\bdot}( (m, \hat 1), N)$ denotes the reduced homology of the order complex of the interval $(m,\hat 1)$ with coefficients in the $k$ module $N$, see  \cite{wachs2006poset}.   When $m= \hat 1$, the module $\Tor^{\widehat P}_j(S_{\hat 1},  S_{\hat 1} \otimes_k H^i( Y, Y - Z_{\hat 1}, \cF ) )$ vanishes for $j > 0$ and equals $H^i(Y, Y-Z_{\hat 1}, \cF)$ for $j = 0$.  
		\end{thm}
		We give a more detailed description of this spectral sequence and a variant involving $H^i(Y - Z_m, \cF)$  in Theorem \ref{FunctorialSS}. In particular, we describe the way in which it is functorial for maps of spaces $Y' \to Y$.

		To explain the meaning of  $\Tor^{\widehat P}_j(S_{\hat 1},  S_m \otimes_k H^i( Y, Y - Z_m, \cF ) )$ without reference to the order complex of ${\widehat P}$,  we need the notion of a ${\widehat P}$ representation and of the derived tensor product of ${\widehat P}$ representations.       
	       We consider ${\widehat P}$ as a finite category: the objects of ${\widehat P}$ are the elements of ${\widehat P}$, and there is a unique morphism $m \to l$ whenever $m \leq l$.  Then a \emph{${\widehat P}$ representation} is a functor ${\widehat P} \to \Mod k$.  Given a representation $M$ of ${\widehat P} \op$ and a representation $N$ of ${\widehat P}$, we may form their tensor product,  which is a $k$ module $M \otimes_{\widehat P} N$.    Then  $\Tor_j(M,N) = H_j(M \otimes^\rL_{\widehat P} N)$ is the homology of the derived tensor product of $M$ and $N$.   The ${\widehat P} \op$  representation  $S_{\hat 1}$ is defined by $l \mapsto k$ if $l =1$ and $l\mapsto 0$ otherwise. 
	      The ${\widehat P}$  representation  $ S_m \otimes_k H^i( Y, Y - Z_m, \cF )$ is defined by $l \mapsto H^i( Y, Y - Z_m, \cF ) $ if $l = m$  and $l \mapsto 0$ otherwise.   These definitions play a important role in our proof of Theorem \ref{vagueFISS}, and we give a more detailed introduction to them in Section 2.1.  
		
		
		To study $\Conf_n(X)$, we specialize to the case where  $Y = X^n$ and ${\widehat P} = \rP(n)$, the lattice of partitions of $n$.   For a partition $l$ of $n$,  define $Z_l$ to be the set of tuples $(x_i)_{i =1}^n$, such that  if $i,j$ are in the same block of $l$, then $x_i = x_j$.  When $X$ is Hausdorff, $Z_l$ is closed.   Then we have $ Y - \bigcup_{l \in P} Z_l = \Conf_n(X)$.   
		  Assume that $k$ is a field.   Let $l$ be a partition $n = b_1 \sqcup \dots \sqcup b_r$.  Using the homology of the partition lattice \cite{wachs2006poset}, we have  $$\Tor^{\widehat P}_{n-r}(S_{\hat 1},  H^i( Y, Y - Z_p, k ) ) ~\iso  ~H^i (X^n , X^n - \Delta_l X^{r}) ^ {\oplus{c_l}}$$  and all the other terms vanish. Here   $\Delta_l: X^{r} \to X^n$ is the diagonal embedding corresponding to $l$, and $$c_l = (\#b_1 -1)!  \cdots  (\# b_r - 1)!.$$
		When $X$ is an orientable manifold of dimension $d$,  we have $$H^\bdot (X^n , X^n - \Delta_l X^{r}, k) ~=~ H^i(X^r, \sqtimes_{i = 1}^r ~\omega_X^{*  \otimes (\# b_i - 1)}) ~=~   H^{\bdot - \sum_i (b_i -1) d } (X^r, k).$$ In this way, the homology groups of our spectral sequence resemble those of Totaro.  We examine further the action of $\FI$ on the $E_1$ page in Section 3.    
  
		
		\subsection{Lattice homological algebra}
		
		We view the spectral sequence of Theorem \ref{vagueFISS} as a shadow of a more fundamental relationship between  $\{H^i(Y, Y - Z_p, \cF)\}_ {p \in {\widehat P}}$ and  $H^i (Y - \bigcup_{p \in P} Z_p , \cF)$, which we describe in this section.    
		
		If  $l \leq m$, then $Y - Z_l$ contains  $Y - Z_m$, and we obtain a restriction map $H^i(Y, Y - Z_l, \cF) \to H^i(Y, Y-Z_m, \cF)$.  These maps give $l \mapsto H^i(Y, Y - Z_l, \cF) $ the structure of a ${\widehat P}$ representation. 
		We lift these representations 
		  from $k$ modules to chain complexes of sheaves on $Y$,  as follows.    Define $$i_{p*}\rR i_p^! \cF  :=  \cone(\cF  \to \rR j_{p*} \cF|_{Y - Z_p})[-1],$$ where   $i_p$ is the inclusion of $Z_p$ into $Y$, $j_p$ is the inclusion of $Y - Z_p$, and $\rR j_{p*}$ is the derived pushforward.   Then $p \mapsto i_{p*}\rR i_p^! \cF$ is a chain complex of sheaves of ${\widehat P}$ representations.  More precisely,   there is an object $i_{\bdot *} \rR i_\bdot^! \cF$ of $\bD(\Sh(X)^P)$ which is isomorphic to $   i_{p*}\rR i_p^! \cF$ in degree $p$. 
		  The hypercohomology groups of this object are the $P$ representations $\{H^i(Y, Y - Z_\bdot, \cF)\}_i$.
		
		Similarly, write $j$ for the inclusion of $Y - \bigcup_{p \in P} Z_p$ into $Y$.  Then $\rR j_* j^* \cF$ is a chain complex of sheaves on $Y$ whose hypercohomology groups are $\{H^i (Y - \bigcup_{p\in P} Z_p , \cF)\}_{i}.$ 
		
		 The following theorem expresses the relationship that underlies Theorem \ref{vagueFISS}:
	
		 \begin{thm}[Resolution of $\rR j_*j^*\cF $]\label{resolution}
			      We have  $$\rR j_* j^* \cF   \iso S_{\hat 1} \otimes^{\rL}_{\widehat P}   i_{\bdot*} \rR i_\bdot^!  ~\cF$$  and applying $\rR \pi_*$ this isomorphism becomes $$~ \rR \pi_* ~ \cF|_U \iso  S_{\hat 1}  \otimes_{\widehat P}^\rL \rR  \pi_*  \rR i_\bdot^! ~\cF.  $$
		\end{thm}
		
		Although the statement of Theorem \ref{resolution} is clean, throughout the body of the paper we avoid the derived category in order make our spectral sequences functorial.   Thus in Theorem \ref{coolres},  we construct a specific zig-zag of quasi-isomorphisms between chain complexes whose images in $\bD(X)$ are $\rR j_* j^* \cF$ and $ S_{\hat 1} \otimes^{\rL}_{\widehat P}   i_{\bdot*} \rR i_\bdot^!  ~\cF$ respectively.  
		
		One interpretation of Theorem \ref{resolution} is as follows.   Given a chain complex $C$ of ${\widehat P}$ representations that computes $H^i(Y, Y - Z_p, \cF)$ and a free ${\widehat P} \op$ resolution $F \simto S_{\hat 1}$,  the action of ${\widehat P}$ on $C$ allows us to build a chain complex $F \otimes_{\widehat P}  C$  that has homology $H^i(X - \cup_{{\widehat P} \in {\widehat P}}, \cF|_U)$. \footnote{ When we say $C$ computes   $H^i(Y, Y - Z_p, \cF)$, we mean that $C$ is quasi-isomorphic as a complex of $P$ representations to the canonical Godement complex.  It is not sufficient for its homology to be isomorphic to $H^i(Y, Y - Z_p, \cF)$. }

		Our original motivation for Theorem \ref{resolution} was to give a global version of the  Goresky--MacPherson formula for the cohomology of the complement of a subspace arrangement \cite{goresky1988stratified}.  We describe how Theorem \ref{resolution} specializes to this formula in Example \ref{GMformula}.

	   	To obtain Theorem \ref{vagueFISS} from Theorem \ref{resolution}, we construct a spectral sequence that converges to the homology of a derived tensor products over ${\widehat P}$. Let   $A$ be a  ${\widehat P} \op$ representation and $M$ be a complex of ${\widehat P}$ representations.  For any  projective resolution  $F \simto A$ we construct a filtration of $F \otimes_{\widehat P} M$, which gives a spectral sequence:  
		
		\begin{prop}[Spectral sequence for $A \otimes^L -$] \label{vagueSS}
			Let $A$ be an object of $\Mod k^{\widehat P}$ and let $M$ be a left bounded  chain complex of ${\widehat P}$ representations in $\Sh(Y)$.   Then there is a spectral sequence of sheaves on $Y$ converging to $H^\bdot(A \otimes_{\widehat P}^\rL M)$, whose $E_1$ page consists of sums of modules of the form $ \Tor_j^{\widehat P}(A, H^i( M_p) )$, in cohomological degree $i-j$.  
		\end{prop}
		
		We give more details in Proposition \ref{SS}, and we  describe the way in which the spectral sequence is functorial.  Heuristically, Proposition \ref{vagueSS} expresses that  since $M$ may be constructed from $H^i(M_p)$ by taking cones and sums,  $A \otimes_{\widehat P}^\rL  M$ may be constructed from $A \otimes^\rL H^i(M_p)$ by taking cones and sums.  After we pass to homology, this construction assembles into a spectral sequence.  
		

	\subsection{Relationship to other work}
		 Derived tensor products are closely related to homotopy colimits, and homotopy colimit  approaches to analogues of the Goresky--MacPherson formula were studied in  \cite{ziegler1993homotopy}.  The homotopy colimits there take place at the level of topological spaces and describe the homotopy type of the union of the closed subsets; information about the cohomology of the open complement is obtained by Alexander duality.    An equivariant version of the Goresky--MacPherson formula  is given in \cite{sundaram1997group}, an arithmetic version is proved in \cite{bjorner1997subspace}, and a version for local cohomology in commutative algebra is found in \cite{montaner2003local}.   Theorem \ref{resolution} is related to these results, but we do not believe that it has been stated before.   
	
		  In \cite{petersen2016spectral},  Petersen gives a spectral sequence that converges to the Borel--Moore homology or compactly supported cohomology of a sheaf $\cF$ on the top stratum of a stratified space, which relates to the work of Getztler on the compactly supported cohomology of configuration spaces\cite{getzler1999resolving}   While Petersen gives results on the behavior of the Borel--Moore homology for configuration spaces of $X$, we study ordinary cohomology.   After establishing the results in this paper, we have found that our situation can be related to Petersen's by Verdier duality, and so Theorem  \ref{resolution} and \ref{vagueFISS} can be established using the arguments in \cite{petersen2016spectral}.  However, they were not known before our work, and we believe the spectral sequence of Theorem \ref{vagueFISS} gives the first technique for computing the cohomology of $\Conf_n(X)$ for a general Hausdorff topological space $X$.  

		The material in Section \ref{backgroundposet} is not original.   We cover it in detail because we want to establish conventions for the derived functors of tensor products over posets, which is a bit subtle for $\Sh(Y)$ since it does not have enough projectives.  We avoid working in the derived category to make our spectral sequences functorial and ensure a good theory of derived tensor products.  

		  In Section 2, we work in the generality of a topological space $Y$ with closed subsets indexed by a meet lattice ${\widehat P}$, because we are interested in applications to global arrangements associated with other families of lattices and other indexing categories.  Some interesting generalizations of $\FI$ have been introduced.   For instance, Gadish  defines the notion of categories of $\FI$-type and studies subspace arrangements indexed  by $\FI$-type categories \cite{gadish2016representation}.  And the $\FI_W$ categories of \cite{wilson2014fiw} have been used to study arrangements with different structure posets in \cite{bibby2016representation}. 
		  
 	\subsection{Questions}
 	 	 What are the generators and relations for the $\FI$ modules that appear in $H^i(\Conf_\bdot(X), k)$, and when does stability occur?    We have recently proved that for any topological space $X$ satisfying the Kunneth formula, the spectral sequence for $\Conf_n(X \times \R^1)$ degenerates at the $E_1$ page, and plan to demonstrate this in future work.  Apart from this case, we have not found explicit bounds on the stable range.  
 	 	 
 	 	   Does there exist a finite $CW$ complex $C$ and $j > i \in \bbN$ such that $H^{j}(\Conf_\bdot C)$ is nonzero and finitely generated, but $H^i(\Conf_\bdot C)$ is 0?   Does there exist a  $C$ such that $H^{0}(\Conf_\bdot C) \oplus H^1(\Conf_\bdot C)$ is finitely generated, but $H^{2}(\Conf_\bdot C)$ is not finitely generated?
 	 	 

	\subsection{Acknowledgments}
		I thank Daniel Barter, Trevor Hyde, Andrew Snowden, Jeremy Miller, Bhargav Bhatt, and Nir Gadish for helpful conversations.  I especially thank John Wiltshire-Gordon for motivation, and for suggesting the reformulation in terms of the dualizing complex.  
		
	\section{Lattice Spectral Sequences} \label{lss}

		\subsection{Background on poset homological algebra} \label{backgroundposet}

			Let ${\widehat P}$ be a finite poset with top element $\hat 1$, and let $k$ be a commutative ring. We write $P$ for $\widehat P - \hat 1$. We consider ${\widehat P}$ as a category: its objects are the elements of ${\widehat P}$, and there is a unique map from $p \in {\widehat P}$ to $q \in {\widehat P}$ whenever $p \leq q$.   We will  study representations of ${\widehat P}$ in the category $\Sh(X)$ of sheaves of $k$ modules on a topological space $X$.    When $X$ is a point, we recover representations of ${\widehat P}$ in $\Mod k$.  We will write $\ku$ for the locally constant sheaf corresponding to $k$. 
			
			\begin{defn}\label{Pcat}
				 A \emph{representation of ${\widehat P}$ in $\Sh(X)$}, $M$,  is a functor from ${\widehat P}$ to $\Sh(X)$.  For each $p \in {\widehat P}$, $M$ gives a sheaf $M_p$, and for each $p \leq q$, a map  $M_{pq}:  M_p \to M_q$ such that if $p \leq q \leq r$ we have $M_{pq} M_{qr} = M_{pr}$, and $M_{pp} = \id$.  We will also say that $M$ is a ${\widehat P}$ \emph{module}, or an object of $\Sh(X)^{\widehat P}$.  We write $\Ch(\Sh(X))^{\widehat P} = \Ch(\Sh(X)^{\widehat P})$ for the category of chain complexes of ${\widehat P}$ representations in $\Sh(X)$.
			\end{defn}

			\begin{defn}[Free modules]
				The \emph{free module} on an object $p$ of ${\widehat P}$ is denoted $\ku {\widehat P}(p, -)$.  We have $\ku {\widehat P}(p, -)_q := \ku {\widehat P}(p,q)$, the  locally constant sheaf of free $k$ modules on the set of maps from $p$ to $q$.  When $p \leq q$,  we have $\ku {\widehat P}(p,q) = \ku$ and $\ku {\widehat P}(p,q)= 0 $ otherwise.  
			\end{defn}
			
			Free modules satisfy the following universal property.  We write it in terms of the dual poset ${\widehat P} \op$, because we will use it most often in this context:  
			
			\begin{ex}[Yoneda]
				Let $N$ be a ${\widehat P}\op$ representation. Then $\Hom_{{\widehat P}\op} ({\ku} {\widehat P}(-, p), N) = N_p$.  In particular, when $p \leq q$ we refer to the map $\ku {\widehat P}(-, p) \to \ku {\widehat P}(- ,q)$ corresponding to $p\leq q \in {\ku} {\widehat P}(p,q)$ as \emph{multiplication by} $p\leq q$.  If we apply $\Hom_{{\widehat P} \op}(-, N)$  we get the map $N_{pq}: N_q \to N_p$ which is the action of $p \leq q$.  
			\end{ex}
			
			\begin{defn}
				Let $S_q$ denote the representation of ${\widehat P}$  defined by $(S_q)_p = \underline {k}$ for  $ p = q$ and $(S_q)_p = 0$ otherwise.   The same definition gives a representation of ${\widehat P} \op$, and we will refer to these representations by the same name.    When $X$ is a point and $k$ is a field, the $\{S_q\}_{q \in {\widehat P} }$ are exactly the simple representations of ${\widehat P}$.

			\end{defn}
			
			Next we define the tensor product operation, which takes a ${\widehat P} \op$ representation and a ${\widehat P}$ representation and produces a sheaf on $X$.  
			\begin{defn}[Tensor products]\label{posettensor}
				Let $M$ be a representation of ${\widehat P} \op$ in $\Sh(X)$, and let $N$ be a representation of ${\widehat P}$ in $\Sh(X)$.  Then the tensor product is  the coequalizer $$M \otimes_{\widehat P} N =  { \rm coeq } (\bigoplus_{p_1 \leq p_2} M_{p_2} \otimes N_{p_1}  \coeq{ M_{p_2p_1} \otimes 1}{1\otimes N_{p_1 p_2}}  \bigoplus_{p \in {\widehat P}}  M_p \otimes N_p)$$ where the first map takes the summand $M_{p_2} \otimes N_{p_1}$ to $M_{p_1} \otimes N_{p_1}$ by multiplication on the right, and the second takes $M_{p_1} \otimes N_{p_2}$ to $M_{p_2} \otimes N_{p_2}$ by   multiplication on the left.  If $M, N$ are chain complexes of ${\widehat P} \op, {\widehat P}$ representations, then their tensor product is a chain complex, where the symbol $\otimes$ in the above definition is interpreted as the tensor product of $\Ch(\Sh(X))$. 
			\end{defn}
			
			\begin{ex}[Co-Yoneda]
				Tensoring with a free module on an element $q$ of ${\widehat P}\op$ is the same as evaluating at that element: $\underline {k} {\widehat P}(-, q) \otimes_{\widehat P} N = N_q$. 
				If we apply $- \otimes_{\widehat P} N$ to the the multiplication by $p \leq q$ map $\ku {\widehat P}(-,p) \to \ku {\widehat P}(-, q) $  we get the map $N_{pq}: N_p \to N_q$, the action of $p \leq q$.  
			\end{ex}
			
			The previous example and the right exactness of $- \otimes_{\widehat P} N$ gives a way of computing tensor products that is often easier than Definition \ref{posettensor}.
			\begin{ex}[Cokernels from presentations]
				If we have a presentation of the ${\widehat P} \op$ module $M$ as the cokernel of a map of free modules $$f: \bigoplus_p V_p \otimes \ku {\widehat P}(-, p) \to \bigoplus_q V_q \otimes \ku {\widehat P}(-, q),$$ determined by maps $f_{pq}: V_p \to V_q$ for $p \leq q$ ,  then $$ M \otimes_{\widehat P} N = \coker(\bigoplus_p V_p \otimes N_p \to \bigoplus_q V_q \otimes N_q),$$ with the map $V_p \otimes N_p \to V_q \otimes N_q$ given by $f_{pq} \otimes N_{pq}$ if $p \leq q$ and by $0$ otherwise. 
			\end{ex}
	
			Many functors are given by tensoring with a ${\widehat P} \op$ representation.  For instance, when $X$ is a point, all left adjoint functors $\Mod k^{\widehat P} \to \Mod k$ arise uniquely in this way. Colimits are another example:
			
			\begin{ex}[Colimits as tensor products]
				Let $\underline {k} \{*\}_{P}$ denote the representation of ${\widehat P} \op$ which is $0$ at $1$,  $\underline {k}$ at every $p < \hat 1$, and for which every map $p \leq q$ for $p,q > 0$ acts by the identity. Then $\colim_{P} M = \underline {k} \{*\}_{P} \otimes_{\widehat P} M.$
			\end{ex}

			\begin{defn}[Pullback from a point]
				Let $A$ be a representation of ${\widehat P} \op$ in $k$ modules. Then for $X$ a topological space, we define $\underline {k} A $ to be $ \underline {k} \otimes_{{k}}  A$,  in other words the \emph{pullback} of $A$ to $\Sh(X)$.  Pullback replaces each $k$ module in the representation by the corresponding locally constant sheaf.  
			\end{defn}
			
			We use the derived functors of $\underline {k} A \otimes_{{\widehat P}} - $ for  $A$ in $\Sh(X)^{\widehat P \op}$.  In the next defintion we give our conventions for derived tensor products.   In summary, although $\Sh(X)^{\widehat P}$ does not have enough projectives,  $\ku A$ admits a resolution by free modules, which we use to define $\ku A \otimes_{\widehat P}^\rL -$.  
			
			\begin{defn}[Derived functors of $\underline {k} A \otimes -$]
				Any resolution of $A$ by free ${\widehat P} \op$ representations, $F \simto A$, yields a resolution $ \underline {k} F \simto   \underline {k} A$.   For $M$ in $\Ch(\Sh(X)^{\widehat P})$, we say that $\ku F \otimes_{\widehat P} M$ is a \emph{model} for $\ku A \otimes_{\widehat P}^L M$. If $B$ is a ${\widehat P}$ representation in $k$ modules, we write $\Tor^{\widehat P}_i(A, B)$ for the $k$ module $H_i(F \otimes _{\widehat P} M) $ 
				
				 Given any two free resolutions $F,F'$ of $A$, there exists a quasi-isomorphism $x: F \to F'$ that extends the identity map on $A$, which is unique up to homotopy.  Let $x'$ be a homotopy inverse to $x$.  Then  $ \ku x \otimes 1:  \ku F \otimes_{\widehat P} M \simto  \ku F' \otimes_{\widehat P} M $ has a homotopy inverse $\ku x' \otimes 1$, and in particular is a quasi-isomorphism.  
				 
				 Given a map of posets $g:  Q \to {\widehat P}$, a free resolution $F \simto A$ over ${\widehat P}$ and a free resolution $G \simto g^* A$ over $Q$, we have a quasi-isomorphism $g^*F \simto g^* A$, and hence a map $j: G \to g^*F$ extending the identity on $g^*A$,  which is unique up to homotopy.  Thus given  $M,N$  objects  of $\Ch(\Sh(X))^{\widehat P}, \Ch(\Sh(Y))^Q$,  and a map $x: N \to g^* M$ we obtain a map $\ku G \otimes_Q N \to \ku F  \otimes_{\widehat P} M$, representing $\ku A \otimes_Q^L N \to \ku A \otimes_{\widehat P}^L M$,  given by the composition $$ \ku G \otimes_Q N  \to^{1 \otimes x}  \ku G \otimes_Q G^* M  \to^{1 \otimes j} \ku g^* F \otimes_Q g^*M \to \ku F \otimes_{\widehat P} N . $$ For any other two resolutions $F', G'$, and comparison maps $j': G' \to g^*F'$,  $x: F \to F'$ and $y : G \to G'$, we have that $j  \circ y$ is homotopic to $g^* x \circ j$,  by the uniqueness of $j$ up to homotopy. 
				 %
				 \QED
			\end{defn}

			 The next example suggests how the homology groups $\Tor_\bdot^{\widehat P}(S_{\hat 1}, S_p)$ are related to configuration spaces.   
			 \begin{exa}[Partition lattice]\label{partition}
				For ${\widehat P} = \rP(3)$ the partition lattice on $3$ elements, the top element $\hat 1$ is the discrete partition $\{ 1\}\{2\}\{3\}$. We have a resolution of $S_{\{1\}\{2\}\{3\}}$ given by $$ {k} {\widehat P}(-, \{1\}\{2\} \{3\}) \leftarrow^{d_0 }  {k} {\widehat P} ( -, \{1\}\{23\}) \oplus {k} {\widehat P}( -, \{2\}\{13\}) \oplus {k} {\widehat P}( -, \{3\}\{12\}) \leftarrow^{d_1} {k} {\widehat P}(-, \{123\}) \otimes {k}^2.$$ 
				 The differential $d_0$ takes the summand $k\widehat P(-, \{r\}\{st\})$ to $k\widehat P (-, \{123\})$ by multiplication by $\{r\}\{st\} \leq \hat 1$.  In other words, it is given by the row vector:  $$d_0 = \onethree{\{1\}\{23\} \leq \hat 1}{\{2\}\{13\} \leq \hat 1}{\{3\}\{12\} \leq \hat 1}.$$ The differential $d_1$ is given by the matrix:  $$\threetwo{\{123\} \leq \{1\}\{23\}}{0}{-\{123\} \leq \{2\}\{13\}} {\{123\} \leq \{2\}\{13\}}{0}{-\{123\} \leq \{3\}\{12\}}. $$  
				 
				 We note that the $0$th free module in this resolution has rank $1$, the first has rank $3$ and the second has rank $2$.  We may assemble this data into a Poincar\'e polynomial $1 + 3 q + 2 q^2$, which agrees with the Poincar\'e polynomial of $\Conf_3(\bbR^2)$. 
				 
				 The symmetric group $\bS_3$ acts on ${\widehat P}$, so we may consider $\widehat P$ modules with an $\bS_3$ action that is compatible with the action on $\widehat P$.  If we let $\bS_3$ act trivially on $S_{\{1\}\{2\}\{3\}}$ we may upgrade the free resolution to one that is equivariant for the $\bS_3$ action.  Then $\bS_3$ acts trivally on the generators of the $0$th free module;  the generators of the first free module are the permutation representation of $\bS_3$; and the generators of the second free module are the standard representation of $\bS_3$.   This  corresponds to the equivariant Poincar\'e polynomial $s_3 + (s_3 + s_{2,1}) q +  s_{2,1} q^2$, which agrees with the equivariant  Poincar\'e polynomial of $\Conf_3(\bbR^2)$. 
				
				To obtain a model for $S_{\{1\}\{2\}\{3\}} \otimes^\rL_{\rP(3)} S_{\{123\}}$ from this free resolution, we tensor it with $S_{\{123\}}$.  All of the terms in the free resolution except for the last become zero, and we have that $\Tor_2(S_{\{1\}\{2\}\{3\}}, S_{\{123\}}) = k^2$, or equivariantly the standard representation $(2,1)$.  
				\QED
			\end{exa}

			The next example shows that a uniform resolution of $S_{\hat 1}$ exists.  

			\begin{exa}[Reduced bar resolution and strict functoriality]\label{barresolution}
				Let $C^s_p$ be an object of $\Ch(\Sh(X)^P)$.  We have a uniform construction of $\ku S_{\hat 1} \otimes^\rL_{\widehat P} C$ coming from the  reduced bar resolution of $S_{\hat 1}$:  $${k} {\widehat P}(-, 1)  \leftarrow \bigoplus_{p_1 < \hat 1}  {k} {\widehat P}(-, p_1) \leftarrow \bigoplus_{p_2 < p_1 < \hat 1} {k} {\widehat P}(-, p_2) \leftarrow \dots    $$
				The $i$th term of this resolution, $\bar B_i$, is the sum over all strict length $i$ chains:  $$\bar B_i = \bigoplus_{p_i < p_{i-1} < \dots < p_1 < \hat 1} {k} {\widehat P}(-, p_i)$$  and the differential $d_i = \sum_{i = 0}^r  (-1)^r\delta_r^i$ comes from a semi-simplical $k$ module.  When $r = 0$, define $\delta_r = 0$.  When $ 0< r < i$, the map $\delta^i_r :\bar B_i \to \bar B_{i-1}$   takes summand corresponding $p_i < \dots < p_r < \dots  < \hat 1$  to the summand $p_i < \dots \widehat p_r < \dots < \hat 1$ by the identity map. When $r = i$,  the map $\delta^i_r$ takes the summand $p_i < \dots  < \hat 1$,  ${k} {\widehat P}(-, p_i)$,  to the summand $p_{i-1} < \dots < \hat 1$, ${k}{\widehat P}( - , p_{i-1})$,  by the multiplication by $p_{i-1} \leq {p_i}$
			
			This resolution is exact because it is the reduced chain complex of the simplical $k$ module $B$ ,  with $B_i = \{p_1 \leq \dots \leq p_i \leq 1 \}$.   The bar construction $B$ is contractible, as in \cite{weibel1995introduction} Chapter 8.   
			
			Tensoring with $C$  we get a model for $\ku S_{\hat 1} \otimes^\rL_{\widehat P} C $, the totalization of the bicomplex:
			$$ (s, -t) \mapsto  \bigoplus_{p_t < p_{t-1} < \dots < p_1 < \hat 1} C^s_{p_t}, $$ with differential in the $s$ variable given by $d_C$ and differential in the $t$ variable  $d_t =  \sum_{i = 0}^r  (-1)^r\delta_r^t$.  When $r = 0$, $\delta_0 = 0$;  for $ 0< r < t$, $\delta_r^t$   takes the summand corresponding $p_t < \dots < p_r < \dots  < \hat 1$  to the summand $p_t < \dots \widehat p_r < \dots < \hat 1$ by the identity map; and when $r = t$,  the map $\delta^t_t$ takes the summand $p_t< \dots  < \hat 1$,  $M_{p_t}$  to the summand $p_{i-1} < \dots < \hat 1$, $M_{p_{i-1}}$  by the action of $p_{i-1} \leq {p_i}$
			
			 Take a $f: Q \to {\widehat P}$  a map between posets with top element such that $f \inv(1) = 1$,  $D$  a chain complex of $Q$ representations, and  $x: D \to f^* C$.   We obtain a representative for $S_{\hat 1} \otimes^\rL_Q D \to S_{\hat 1} \otimes^\rL_{\widehat P} C$ from the map $$\bigoplus_{q_r < \dots < q_1 < \hat 1} D^i_{q_r} \to \bigoplus_{q_r < \dots < q_1 < \hat 1} C^i_{fq_r} \to \bigoplus_{p_r < \dots < p_1 < \hat 1} C^i_{p_r,}$$ where the last map takes the summand corresponding to a chain  $q_r < \dots < q_1 < \hat 1$ to $fq_r < \dots < fq_1 < \hat 1$ if $fq_i \neq fq_j$ for any distinct $i,j \in \{1, \dots, r\}$  and takes  the summand to $0$ otherwise.  \QED 
			 
			\end{exa}
			
			We may use the bar complex to compare the derived tensor product with poset homology.  
			
			\begin{prop}[Comparison with poset homology] \label{pohomcompare}
			 	When $q < \hat 1$,  the reduced bar complex for $S_{\hat 1} \otimes^\rL_{\widehat P} S_q [2]$ is isomorphic to the complex of chains $C_*((q, 1))$, see \cite{wachs2006poset}.  In particular, we have  $\Tor_i(S_{\hat 1} ,S_q) = \tilde H_{i-2}((q,1), k)$ where $\tilde H_*$ is the reduced poset homology.  
			\end{prop}

			For specific posets, there  often are smaller resolutions than the bar complexes that give models for $S_{\hat 1} \otimes^\rL - $.

			\begin{exa}[Products]
				Let ${\widehat P}, \widehat Q$ be posets with top elements $\hat 1_{\widehat P},  \hat 1_{\widehat Q}$. Given a representation of ${\widehat P}$ and a representation of $\widehat Q$ we can tensor them over ${k}$ and form a representation of ${\widehat P} \times \widehat Q$.  We denote this tensor product by $\boxtimes$. Then $S_{\hat 1_{{\widehat P} \times \widehat Q}} = S_{\hat 1_{\widehat P}} \boxtimes_{{k}}  S_{\hat 1_{\widehat Q}}$. For any two elements $p,q$ we have $ {k}({\widehat P} \times \widehat Q) (-, p \times q) = {k} {\widehat P}(-,p) \boxtimes {k} \widehat Q(-,q)$.  So from  free resolutions $F \simto S_{\hat 1_{\widehat P}}$ and $G \simto S_{\hat 1_{\widehat Q}}$ we obtain a map $F \boxtimes G \simto S_{\hat 1_{{\widehat P} \times  \widehat Q}}$, which is a quasi-isomorphism because  the underlying $k$ modules are free.  
			\end{exa}
			
			\begin{exa}[Boolean poset]
				Let $B([n])$ be the poset of  subsets of $[n] = \{1, \dots, n\}$ ordered by reverse inclusion. Giving a $B([n])$ representation is the same as giving a $\N^n$ graded ${k}[x_1, \dots, x_n]$ module, concentrated in degrees $\{0,1\}^n$.   We have that $B([n])$ is isomorphic to the $n$ fold product of the two element poset $0 \leq 1$.   For $0 \leq 1$, the simple $S_{\hat 1}$ is resolved by ${k} {\widehat P}(-, 1) \leftarrow {k} {\widehat P} (-, 0)$.  From the above, we get a Koszul resolution of $S_{\hat 1}$ for $B([n])$.   
				
				If $C_I^s$ is a chain complex of $B([n])$ representations, then $S_{\hat 1} \otimes_{\widehat P} C$ is the totalization of the bicomplex:   $$(s,-t) \mapsto  \bigoplus_{I \subset [n], ~ |I| = t}  {j_I}_* C_I^s|_{U_I} $$  with $d_t = \sum_{i = 1}^{n} (-1)^i x_i$ and $d^s = d_C^s$
				 				
			\end{exa}

			\begin{defn}
				We call a free resolution \emph{finite}  when it has finitely many terms and each is finitely generated.  
			\end{defn}
			
			 The following proposition is not necessary for what follows, but the proof gives an algorithm that useful for computing small resolutions.  For the boolean poset and the partition poset, this process gives free resolutions isomorphic to the ones described above.  
			\begin{prop}  [Minimal Free Resolutions]
				When $k$ is a field, every representation $A$ of a finite poset ${\widehat P} \op$ admits a finite free resolution $F \simto A$ such the differential of the chain complex $F \otimes_{\widehat P} S_p$ vanishes for any $p \in {\widehat P}$.		\end{prop} 
			\begin{proof}
				Mimicking the situation of  modules over a graded ring, we may construct a ``minimal'' free resolution.  Let $NA$ denote the vector space, graded b
				y the elements of ${\widehat P}$ such that in degree $p$ we have $$(NA)_p = \frac{A}{\sum_{pq} \im A_{pq}},$$ or in other words $NA =\bigoplus_{p \in {\widehat P}}  A \otimes_{\widehat P} S_p$, where $S_p$ is the simple ${\widehat P}$ rep which is one dimensional at $p$ and zero elsewhere.  Pick a basis for this vector space $\{ e_i^p\}_{p \in {\widehat P}, i = 1, \dots, n_p}$, and then choose representatives $\{a_i^p\}$ for each basis vector in $A$.  These generate $A$, and so we get a surjection $\oplus_{i,p} k {\widehat P}(-, p) \onto A$.  This is the first step in the free resolution.  If $A \neq 0$ then the kernel of this surjection $K_1$  must be zero in at least one degree where $A$ was nonzero (any degree $q$ such that $A_{q'} = 0$ for all $q' \geq q$), and so the process terminates.  
			\end{proof}
			Over a field, there is a dual theory for minimal injective resolutions of poset representations.  
			
			In addition to the derived tensor product with $S_{\hat 1}$, we will use the derived functors of the colimit over $P = \widehat P - \hat 1$.  
			\begin{defn}[Homotopy colimit] 
				 Let $M$ be an object of $\Ch(\Sh(X))^{\widehat P}$. The \emph{homotopy colimit of $M$ over $O$} is is the left derived functor of   $M \mapsto \colim_{{\widehat P} - 1} M|_{P}$.  We denote it by  $\hocolim_{P}  M$.   There is a map $\hocolim_{P} M \to M_1$ given by $\hocolim_{{\widehat P} - 1}M \to \colim_{P} M|_{P} \to M_1$ and we say that the representation $M$ is ${\widehat P}$-\emph{exact}, if this map is a quasi-isomorphism.  
	
			\end{defn}

			\begin{prop}[Comparison with $S_{\hat 1} \otimes_{\widehat P}^L - $] \label{cone}
			 	We have that $S_{\hat 1} \otimes_{\widehat P}^L M \simeq \cone(\hocolim_{P} M \to M_1)$.  More precisely for any free resolution  $F$ of ${k} \{*\}_{P}$ with canonical map $F \otimes_{\widehat P} M  \to \colim_{P} M \to M_1$ , we have that  $\cone(F \otimes_{\widehat P} M \to M_1)$  is a model for  $S_{\hat 1} \otimes_{{\widehat P}}^L M$.  In particular, $M$ is an exact representation if and only if $ S_{\hat 1} \otimes^\rL_{\widehat P} M \simeq 0$.
			\end{prop}
			\begin{proof}
				We have a short exact sequence of ${\widehat P} \op$ representations:  $ 0 \to \underline {k} \{*\}_{P} \to  k {\widehat P}(-, 1) \to S_{\hat 1}  \to 0$.  So for any resolution $F$ of ${k} \{*\}_{P}$, the cone of $F \to {k} \{*\}_{P} \to kP(-,1)$ is a free resolution of $S_{\hat 1}$.  Tensoring with $M$ gives $S_{\hat 1} \otimes^\rL M \simeq \cone(F \to kP(-, 1)) \otimes_{\widehat P} M  = \cone(F \otimes_{\widehat P} M \to M_1)$ 
			\end{proof}

			The next property says that $\rR \pi_*$ takes derived tensor products to derived tensor products, which is what allows us to use them to compute cohomology.  It is a wonderful feature of the derived world that all operations become exact, and so even right adjoints play nicely with colimits.

			\begin{prop}[ Derived pushforward preserves derived tensor products]\label{hocolimcommutes}
				Let $A$ be an object of $\Mod k^{{\widehat P} \op}$, and let $M$ be an object of   $\Ch(\Sh(X))^{\widehat P}$.  Then for a finite free resolution $F \simto A$, and   $M \simto \cJ$ a resolution of $M$ by $\pi_*$ acyclic sheaves, we have that $F \otimes_{\widehat P} \pi_* \cJ = \pi_* \ku F \otimes_{\widehat P} \cJ$, where the right hand side is a model for $\rR \pi_*(\ku A \otimes^\rL_{\widehat P} M)$ and the  left hand side is a model for $  A   \otimes^\rL_{\widehat P} \rR \pi_* M$ 
			\end{prop}
			\begin{proof}
			
				Write $F_i = \bigoplus_{p \in {\widehat P}}  V^p_i \otimes {k} {\widehat P}( -, p)$ where  the multiplicity space $V^p_i$ is a finitely generated free $k$ module, and the differential $F_i \to F_{i-1}$ is given by maps $d_i^{p \leq q}:V_i^p \to V_{i-1}^q$.  Then since $\pi_*$ commutes with finite direct sums, we have that $\pi_* \ku F \otimes \cJ $ is the totalization of the bicomplex $$(-i,j) \mapsto  \bigoplus_{p \in {\widehat P}} V_i^p \otimes \pi_* \cJ_p^j $$ with differential in the $i$ variable given by $\sum_{p \leq q}  d^i_{p \leq q} \otimes \pi_*\cJ^j_{pq}$ and differential in the $j$ variable given by $d_\cJ$.  Since finite sums of  $\pi_*$ acyclics are acyclic, and $\ku F \otimes_{\widehat P} M \to {\ku} F \otimes_{\widehat P} \cJ$ is a quaisisomorphism, this complex is a model for $\rR \pi_*(\ku A \otimes^\rL_{\widehat P} M)$.   And it is equal to $F \otimes \pi_*  \cJ $, which is a model for $A \otimes^\rL_{\widehat P} \rR \pi_*(M)$
			\end{proof}

			If  $Q$ and ${P}$ are meet lattices, then a meet preserving homomorphism $ \widehat Q \to {\widehat P}$ that takes $\hat 1_Q$ to $\hat 1_P$ has the property that derived tensor products pull back to derived tensor products. 
			\begin{prop}[Transfer] \label{transfer}
				Let $G: \widehat Q \to {\widehat P}$ be a surjective meet-preserving morphism of meet-lattices .  Let  $C$ be an object of $\Ch(\Sh(X))^{\widehat P}$ and $A$ of $\Mod k^{{\widehat P}^{op}}$.  Then $G^*A \otimes_{ \widehat Q}^L  G^*C \to A  \otimes^\rL_{\widehat P} C$ is a quasi-isomorphism for any choice of resolution of $A$ and $G^* A$.  In particular  $C$ is ${\widehat P}$-exact if and only if $G^* C$ is $ \widehat Q$-exact. 
			\end{prop}
			\begin{proof}
				It suffices to prove the statement for a particular choice of free resolutions for $A$ and $f^*A$. First we note that the meet preserving map $G:  \widehat Q \to {\widehat P}$ has an adjoint $F:  {\widehat P} \to  \widehat Q$  given by $$Fp =  \wedge_{q', ~ Gq' \geq p }  ~q'.$$  If $Fp \leq q$ then applying $G$ we have $ p \leq \wedge_{q', ~ Gq' \geq p } G q' \leq Gq.$  Conversely, if $p \leq Gq$ then $Fp = \wedge_{q', ~ Gq' \geq p } q' \leq q$.  Further because $G$ is surjective, there is some $q'$ with $Gq' = p$, and hence $GFp = p$.  
				
				Thus  $G^* \underline {k} {\widehat P}(p, -) = \underline {k} {\widehat P}(p, G-) = \underline {k}  \widehat Q(Fp, -)$.  Take a resolution of  $A$ by free modules: $K_i = \bigoplus_{p} V^p_i \otimes  k {\widehat P}(-,p)$  where $V^p_i$ is a finitely generated free $k$ module, and the differentials are given by maps $d_i^{p \leq q}:V_i^p \to V_i^q$.
				This pulls back to a resolution of $G^* A $ by free modules $G^* K_i = \bigoplus_{p} V^p_i \otimes k \widehat Q(-, Fp)$.  Tensoring with  $C$ we get on the one hand a complex representing  $A \otimes_{ \widehat Q}^L C$, which is the totalization of the bicomplex:
				$$ (-i, j) \mapsto  \bigoplus_{p} V^p_i \otimes C^j _p,$$ 
				and differential in the $i$ variable given by $d_i^{p \leq q} \otimes C^j_{pq}$, and the $j$ variable given by $d_C$. On the other hand we get a complex representing $G^* A\otimes^\rL G^* M$ which is the totalization of the bicomplex  $$ (-i, j)  \mapsto  \bigoplus_{p} V^p_i \otimes C^{j}_{GFp} =  \bigoplus_{p} V^p_i \otimes C^{j}_{p},  $$ with the same differentials, and the natural map between them is an isomorphism.  The last statement of the theorem follows because $G^* S_{\hat 1} = S_{\hat 1}$
			\end{proof}
	We note that the above proposition can be seen as an elementary corollary of the Quillen Fibre lemma. 

	\subsection{Spectral sequences}
		
			We first show how to obtain a spectral sequence from a derived tensor product.  Then we prove the boolean case of Theorem \ref{resolution}, and extend it to a full version.  We combine these two statements in Theorem \ref{FunctorialSS}, to obtain  spectral sequences.

				Suppose we have $C$ in $\Ch(\Sh(X))^{\widehat P}$  and we want to compute $H^i(\hocolim_{P} C)$. There are two relevant spectral sequences: we could take the hypercohomology spectral sequence relating $\hocolim_{P} H^i(C)$ to $H^i (\hocolim C)$ , or we could use the canonical filtration on ${\widehat P}$ representations to get spectral sequence relating $H^i(\hocolim_{P} C_p)$ to $H^i(C)$.  In fact,  we obtain a spectral sequence that combines these two, works also for $A \otimes_{\widehat P}^L -$ and is natural both in ${\widehat P}$  and in free resolutions $F \simto A$.  We give the filtration explicitly in the statement of the theorem:

\begin{prop}[Spectral sequence for $A \otimes^\rL_{ P} -$] \label{SS}
				Let ${ P}$ be finite poset, equipped with a function $r: { P} \to \N$ such that $p > p' \implies r(p) < r(p')$.  Let $C$  be a left bounded chain complex of ${ P}$ representations.  Let  $A$ be a ${ P} \op$ representation, with a finite free resolution  $G_j = \bigoplus_p V_j^p \otimes kP (-,p)$, with differentials given by $d_i^{p \leq q}:V_i^p \to V_{i-1}^q$. Then we have a model for $\ku A \otimes_{ P} C$ given by the totalization of the bicomplex $$(s, -t) \mapsto \bigoplus_{p \in { P}} C^{s}_p \otimes V_t^p,$$ with $d_t$ given by $d_t^{p \leq q} \otimes C^s_{pq} :  V_t^p \otimes C^s_p \to V_{t}^q \otimes C^s_{q}$, and $d_s = d_C$.   And the filtration $ F_i (G \otimes_{ P} C) := $
				
				$$ \bigoplus_{{ t\in \N, p \in { P}, ~s \in \Z ,~ 2 r(p) + s < i}} V_t^p ~\otimes ~C^s_p ~ \oplus  \bigoplus_{{ t \in \N, p \in { P} , ~s \in \Z  ,~ 2 r(p) + s = i}}\ker(1~ \otimes ~d^s_{p,C}: V_t^p ~ \otimes ~C^s_p \to V_t^p~ \otimes ~C^{s+1}_p) $$
		  gives a convergent spectral sequence
		   $$E_1^{i, -j} =  \bigoplus_{s,t, ~s - t = i - j} ~  \bigoplus_{p  \in { P}, ~ 2 r(p) = i - s}  \ku \Tor_{t} (A,  H^s(C_p) \otimes S_p)  \implies H^{i-j}( A \otimes^\rL_{ P} C)$$
		   Alternately, for $p \in { P}, s,t \in \Z$, the term $ \ku \Tor_t(A, H^s(C_p) \otimes S_p)$ appears exactly once on the  $E_1$ page,  in the place $E_1^{i, -j} = E_1^{ 2 r(p) + s, - t - 2r(p)}$.

			\end{prop} 
			\begin{proof}

				 This asserted filtration is increasing, and preserved by both differentials since the bar differential only decreases $2r(p) + s$,  $d_C$ increases it by at most one, and if $2 r(p) + s = i$  then $d_C$ acts by $0$ .   The associated graded is $F_{i}G/F_{i-1}G=  $
				 $$  \bigoplus_{{ t \in \N, p \in { P}, ~s \in \Z  ,~ 2 r(p) + s = i -1}}  V_t^p \otimes \frac{C^s_p}{\ker(d^s_{p,C})} \oplus   \bigoplus_{{ t \in \N, p \in { P} , ~s \in \Z  ,~ 2 r(p) + s = i}} V_t^p \otimes {\ker(d^s_{p,C})}.$$ 
				  On this page, if $p < q$,  the map  $ d_t^{p \leq q} \otimes C^s_{pq}$ must act by zero, because it decreases $2 r(p) + s$ by $2$.    
				   So the $d_t$ differential preserves the grading by $p \in { P}$, as does $d_s = d_C$ and so we get a direct sum decomposition of  $F_i/F_{i-1}$ over $p \in { P}$ 
				   as:
				  $$\bigoplus_{p \in { P}} \bigoplus_{{ t \in \N, p \in { P}, ~s \in \Z  ,~ 2 r(p) + s = i -1}} V_t^p \otimes  \cone\left (d_{C,p}^{s-1} : \frac{C^{s-1}_p}{\ker(d_{C,p}^{s-1})} \to \ker(d^s_{p,C})\right ) $$  
				  $$ =    \bigoplus_{p \in { P}}  \cone\left ( \frac{C^{i - 2 r(p)-1}_p}{\ker(d_{C,p}^{i - 2 r(p)-1})}   \to \ker(d^{i - 2 r(p)}_{p,C})\right ) \otimes (G \otimes_{ P} S_p) \to^{\simeq} \bigoplus_{p \in { P}} H^{i - 2 r(p)}(C_p) \otimes (G \otimes_{ P} S_p)$$
				where the last line is the natural map to the cokernel, which is a quasi-isomorphism by the spectral sequence for a bicomplex.     Taking cohomology yields the  desired $E_1$ spectral sequence.  

			\end{proof}
			
		     \begin{remark}
		     		We have the following functoriality for the spectral sequence of \ref{SS}:  let  $f: Q \to { P}$ a map between ranked posets ($r(fq) = r(q)$), $D$ a chain complex of $Q$ representations, $B$ a $Q^	{op}$ module with free resolution $K$, and $u: D \to f^* C$, and $v: B \to f^* A$,  and $\tilde v: K \to f^* A$.   Then the map $K \otimes_Q D \to F \otimes_{ P} C$ preserves the filtration, and gives morphism of spectral sequences  which on cohomology  is $H^{i-j}(B \otimes^\rL_Q D) \to H^{i-j}(A \otimes^\rL_{ P} C)$.  
		       
		        	Because the rank function of $Q$ is non-degenerate we have $f^* S_p = \bigoplus_{q \in Q, fq = p} S_q $, and we get a map $ \bigoplus_{ q \in Q, fq =p} \Tor^{Q}_*(B, S_q) \to \Tor^{{ P}}_*(A, S_p)$ induced by $\tilde v$.  And the map $D \to f^*C$ gives a map $\bigoplus_{q \in Q, ~ fq = p}H^s(D_q) \to   H^s(D_p)$.   The action of $K \otimes_Q D \to F \otimes_{ P} C$ on the $E_1$ page is given by $  \bigoplus_{q \in Q, fq = p}  H^s(C_q) \otimes \ku \Tor^Q_t(A, S_q)   \to H^s(C_p) \otimes \ku \Tor^P_t(A, S_p) $ the diagonal of the tensor product of the two maps.  				\end{remark}

		To formalize the topological setup for our spectral sequences, we define a category $\cX$.   The maps $\cX$ are exactly the maps of spaces and sheaves that give maps of spectral sequences.  

			
		\begin{defn}
			Let $\cal X$ denote the category  whose objects consist of the data $(X, \cF,   { P},  \{Z_p\}_{p \in {\widehat P}})$ where:
			\begin{itemize}
				\item $X$ is a topological space,
				\item $\cF$ is a sheaf on $X$,
				\item   ${ P}$ is a finite meet lattice, with function $r_{\widehat P}: {\widehat P} \to \bbN$ such that $p < q \implies r_{\widehat P}(p) > r_{\widehat P}(q)$,
				\item  $\{Z_p\}_{p \in {\widehat P}}$ is a collection of closed subsets indexed by ${\widehat P}$ such that $Z_{p} \cap Z_{q} = Z_{p \wedge q}$  and $Z_{\hat 1} = X$. 
			\end{itemize}  
			
			A morphism from $(X, \cF,   {\widehat P},  \{Z_p\}_{p \in {\widehat P}})$ to $(Y, \cG, Q, \{Z_q\}_{q \in Q-1})$ is given by:
			\begin{itemize}
				 \item A map $f: Y \to X$ 
				 \item A ranked map $f : Q \leftarrow {\widehat P}$ such that $f\inv(Z_p) = Z_{fq}$,\  
				 \item A map of sheaves  $y: f^* \cF \to \cG$.
			\end{itemize}
			
			  If we have two maps $f: X' \to X$, $f: {\widehat P} \to {\widehat P}'$,  $y: f^* \cF  \to  \cF'$ and $g: X'' \to X'$, $g: {\widehat P}' \to {\widehat P}''$, $y': g^*\cF' \to \cF{'}{'}$ then their composite is given by $f \circ g: X'' \to X$ and  the map $ y' \circ g^*y : (f^* \circ g^*) \cF \simeq g^* f^* \cF \to^{g^*y } g^* \cF' \to^{y'} \cF{'}{'}$, where we use the natural isomorphism $f^* g^* \simeq (g \circ f)^*$.
\QED
			
		\end{defn}
		
		We also fix some notation in this context.
		\begin{defn}
			For an object $(X, \cF,   {\widehat P},  \{Z_p\}_{p \in {\widehat P}})$ of $X$, we write $i_p : Z_p \to X$ for inclusion, and $j_p : U_p \to X$ for the inclusion of $U_p := X - Z_p$.  We  define $U = X - \bigcup_{p \in P} Z_p$ and write $j: U \to X$ for the inclusion of $U$ into $X$. 
		\end{defn}

			Let  $B([n])$  be the boolean poset of all subsets  $ I \subset\{1, \dots, n\}$, ordered by $I \leq J \iff I \supset J$. We prove a version of Theorem \ref{resolution} in the case ${\widehat P} = B(n)$ 
			
			 Recall that restrictions and pushforwards of flabby sheaves along open inclusions are flabby.

			\begin{prop}[Boolean case] \label{booleancase}
			       Let $(X, \cF,   {\widehat P},  \{Z_p\}_{p \in {\widehat P}})$ be an object of $\cX$ with ${\widehat P} = \rB(n)$.   Let $\cF \to \cJ^\bdot$ be a flabby resolution.   Let  $C$  be the chain complex in $\Sh(X)^{\rB(n)}$ given by   $C_I = j_*{\cal J}^{\bullet}|_{U_I}$ for $I \neq \emptyset$ and $C_\emptyset = j_*\cJ|_{U}$, with maps given by restriction.  Then $C$ is $\rB(n)$-exact. 
			\end{prop}
			
			\begin{proof}
				  To show that $C$ is $\rB(n)$-exact, we must show that the totalization of the bicomplex $$(s,-t) \mapsto  \bigoplus_{I \subset [n], ~ |I| = t}  {j_I}_* \cJ^s|_{U_I}, $$ with differential in the $s$ variable given by $d_{\cal J}$ and differential in the $-t$ variable given by the alternating sum of the restriction maps, is exact.  By the spectral sequence for a bicomplex, it suffices to show that for every flabby sheaf $\cI$ on $X$, the cochain  complex $$C^{-t} = \bigoplus_{I \subset [n], ~ |I| = t}  {j_I}_* \cI|_{U_I} ,$$  with differential the alternating sum of the restriction maps, is exact.  We do this by induction on $n$.  In the case $n=1$, we have $U_{\emptyset} = U_{\{1\}},$ and so the complex has two terms with its only differential the identity, and hence is exact.  When $n > 1$ we have that $C^\bdot$ is the pushforward of a complex of acyclics from $U_{[n]} = \cup_{i \in [n]} U_i$, so we are reduced to showing the complex is exact there, and hence to showing that $C^\bdot |_{U_i}$ is exact for all $i$.    Using that $j^*j_* \simeq id$, and that $U_I \cap U_i = U_i$ if $i \in I$ we have that    
				 $$ C^{-t}|_{U_i} = 
				 \bigoplus_{I \subset [n], ~ |I| = t}  {j_I}_* \cI|_{U_I \cap U_i}  = \bigoplus_{J \subset [n], ~ i \in J,~ |J| = t} \cI |_{U_i} ~\oplus ~  \bigoplus_{K \subset [n] - i, |K| = {t-1}}  {j_i}_* \cI|_{\bigcup_{k \in K} (U_k \cap U_i)}.  $$ 
				 The differential preserves the summand corresponding to subsets $K$ not containing $i$, so a we get a two step filtration of $C^\bdot|_{U_i}$, with associated graded pieces corresponding to the two summands above.  The first complex is an iterated cone of the the identity map of $\cI|_{U_i}$, and so is exact.\footnote{Alternately it comes from a constant representation of the boolean poset of subsets containing $i$, and so is exact.} And by induction applied to $[n] - \{i\}$ on $X = U_i$, the second summand is the pushforward of an exact complex of acyclics, and hence is exact.
			\end{proof}

		The following theorem is the source of our spectral sequences.  It gives a more formal refinement of Theorem \ref{resolution}
		\begin{thm}[Resolution]\label{coolres}
			Let $(X, \cF,   {\widehat P},  \{Z_p\}_{p \in {\widehat P}})$ be an object of $\cX$.   Let $\cF \simto \cJ$ be a flabby resolution. Define  $j_{*\bdot} \cJ|_{U_\bdot}$ to be object of $\Ch(\Sh(X))^{\widehat P}$ given by $p \mapsto j_{*p} \cJ|_p$, and $\cJ_\bdot$ to be the constant representation $p \mapsto \cJ$.   Define $i_{*\bdot} i_\bdot^! \cJ$ to be $\cone(\cJ_\bdot \to j_{* \bdot} \cJ|_{U_\bdot})$.  Then for any free resolution $G \simto k\{*\}|_{{\widehat P} -\hat 1}$ we construct a quasi-isomorphism
			$$ j_* \cJ|_{X - \bigcup_{p \in P} Z_p}  \simto   \hocolim_{P}  j_{*\bdot} \cJ|_{U_\bdot} $$
			 and a zig-zag of quasi-isomorphisms
			$$  j_* \cJ|_{X - \bigcup_{p \in P} Z_p}  \simto \leftsimto S_{\hat 1} \otimes_{\widehat P}^{\rL} i_{*\bdot} i_\bdot^! \cJ. $$
		  Thus for any object $C_\bdot$ of $\Ch(\Mod k^{\widehat P})$ that is isomorphic to $\pi_* i_\bdot^! \cJ$ in ${ \normalfont \textbf{D}}(\Mod k^{\widehat P})$  we have  $H^i(S_{\hat 1} \otimes_{\widehat P}^\rL C_\bdot) \iso H^i(X - \bigcup_{p \in P} Z_p, \cF)$.
			
		\end{thm}
		\begin{proof}
				For the first  quasi-isomorphism, we show that the representation defined by $p \mapsto j_{*p} \cJ|_{U_p}, p \in P$ and $1 \mapsto \cJ$ is exact. Since $\rB(n)$ is the free meet lattice, have a surjection from  $\rB(n) \to {\widehat P}$ where $n$ is the number of atoms of ${\widehat P}$.   Proposition \ref{transfer} reduces  its exactness  to that of its pullback to $\rB(n)$, and this is the case we showed in Proposition \ref{booleancase}.   Fix a free resolution  $G \simto k \{ *\}_{P}$.   Exactness implies that the map $\phi: G \otimes_{\widehat P}  j_{*\bdot} \cJ|_{U_\bdot} \to \colim_{P} j_{*\bdot} \cJ|_{U_\bdot} \to j_*\cJ|_{U} $ is a quasi-isomorphism.  
				
				  Let $F = \cone(G \to kP(-, 1))$.  The map $F \simto S_{\hat 1}$ gives a free resolution of $S_{\hat 1}$.  We have that
				  	 $$F \otimes_{\widehat P} j_{*\bdot} \cJ|_{U_\bdot} = \cone( G \otimes_{\widehat P} j_{*\bdot} \cJ|_{U_\bdot} \to kP(-,1) \otimes_{\widehat P}  j_{*\bdot} \cJ|_{U_\bdot}) \simto \cone( j_* \cJ|_U \to 0) = j_*\cJ|_U[1].$$
				  Therefore 
				  	$$F \otimes_{\widehat P} i_{*\bdot} i_\bdot^! \cJ  = \cone( F \otimes_{\widehat P}  \cJ_\bdot  \to  F \otimes_{\widehat P} j_{*\bdot} \cJ|_{U_\bdot})[-1] \leftsimto \cone( 0 \to  j_*j^*\cJ[1])[-1] =  j_* \cJ|_U. $$
				The quasi-isomorphism $\leftsimto$  follows because the derived tensor product of $S_{\hat 1}$ with any constant $\widehat P$ representation is zero.  
		\end{proof}

		\begin{ex}[Goresky--MacPherson formula]\label{GMformula}
			Suppose $Y = \R^{n}$, the closed subsets $Z_p$ are linear subspaces,  of codimension $r_p$, and $\cF$ is the constant sheaf $\Z$.  Then by Theorem \ref{coolres},  $\rR \pi_* \Z|_{U} \simeq S_{\hat 1} \otimes_{\widehat P}^\rL C_\bdot $, where the chain complex of ${\widehat P}$ representations $C_\bdot$, has cohomology $H^\bdot(C_p) \simeq H^\bdot(\R^n, \R^n - Z_p, \Z) = \Z[ - r_p]$.    
			
			We claim that $C_\bdot$ is isomorphic to a direct sum of its homology groups in the derived category. 
			Filter $C$ by rank, so that $(F_i C)_p = C_p$ if $r_p \leq i$ and  $(F_i C)_p = 0$ otherwise.  Then $F_{i} C / F_{i-1} C \iso \bigoplus_{p, ~ r_p = i} \Z[-r_p]$.  For $i > j$  we have that $\Ext^1_{\widehat P}(F_{i} C / F_{i-1} C, F_{i} C / F_{i-1} C) = 0$.   Therefore the filtration splits and the claim follows. Thus for $i > 0$ we have
			
			$$ H^i(U, \Z)~ =~ H^i(S_{\hat 1} \otimes_{\widehat P}^\rL ~\bigoplus_{p \in {\widehat P}}~  \Z[-r_p]) ~=~ \bigoplus_{p  \in {\widehat P}} \Tor_{-i + r_p}^{\widehat P}(S_{\hat 1}, S_p) $$ 
			Here $\Tor_d^{\widehat P}(S_{\hat 1}, S_p) = H_d(S_{\hat 1}\otimes^\rL S_p)$, where $S_p$ in $\Ab^{\widehat P}$  is defined by $q \mapsto 0$, $q \neq p$,  and $p \mapsto \Z$.   By  Proposition \ref{pohomcompare},  $\Tor_d^{\widehat P}(S_{\hat 1}, S_p) =  \tilde H_{d-2}((p, 1), \Z)$, and we obtain the classical statement \cite{wachs2006poset}.    
		\end{ex}

		Now we construct the spectral sequences that we use to study configuration space.  First, we recall the definition of relative cohomology. 

	 	\begin{defn}[Relative cohomology]
	 		Let $X$ be a topological space, $U$ an open subset and $\cF$ a sheaf on $X$.  For any flabby resolution $\cF \simto \cJ$ we define the \emph{relative cohomology} $H^i(X, U, \cF)$ to be  $H^i(\pi_* \cone(\cJ \to j_* \cJ|_U)[-1])$. 
	 	\end{defn}
	
		\begin{thm}\label{FunctorialSS}
			Let $(X, \cF,   {\widehat P},  \{Z_p\}_{p \in {\widehat P}})$ be an object of $\cX$.
			Then we construct spectral sequences 
		
		  $$E_1^{i, -j} =  \bigoplus_{s,t, ~s - t = i - j} ~  \bigoplus_{p  \in {\widehat P}, ~ 2 r(p) = i - s}   H_{t}(\hocolim_{P} H^s(U_p, \cF ))  \implies H^{i-j}(X - \cup_{p \in P}Z_p, \cF)$$
				 and 
		  	$$E_1^{i, -j} =  \bigoplus_{s,t, ~s - t = i - j} ~  \bigoplus_{p  \in P, ~ 2 r(p) = i - s} \Tor_{t}(S_{\hat 1}, H^{s}(X,  U_p, \cF))  \implies H^{i-j}(  X - \cup_{p \in P} Z_p, \cF), $$
				where by convention, in the second spectral sequence we have $H^s(X,U_1, \cF) := H^s(X, \emptyset, \cF) = H^s(X, \cF)$.   These spectral sequences are natural in the sense that they define a functor from $\cX $ to spectral sequences of $k$ modules.  
		
		\end{thm}	
		\begin{proof}
			To obtain strict functoriality we use Godement's canonical flabby resolution of $\cF$, denoted $C(\cF, X)$ \cite{godement1958topologie}. To compute derived tensor products we use the reduced bar resolution $\bar B(S_{\hat 1})$ and the shift of its truncation that resolves $k\{*\}|_{P}$, which we denote by $\bar B(k\{*\}|_{P})$.  Form the chain complex of  ${\widehat P}$ representations $p  \mapsto \pi_* C(\cF, X)|_{U_p}$, which models $\rR \pi_*(\cF|_{U_p})$ because restrictions of flabby sheaves are flabby.  Then applying the construction of Proposition \ref{hocolimcommutes} to the constructions of Theorem \ref{coolres}, we obtain functorial quasi-isomorphisms $$\bar B(k\{*\}|_{P}) \otimes_{\widehat P}  \pi_* C(\cF, X)|_{U_p} \simto \pi_* (C(\cF, X)|_{U} )$$  and $$\bar B(S_{\hat 1}) \otimes_{\widehat P} \cone(C(\cF, X) \to \pi_* (C(\cF, X)|_{U_\bdot})  \simto \pi_* (C(\cF, X)|_{U}).$$   These  quasi-isomorphisms give functorial spectral sequences by Proposition \ref{SS}.  
		\end{proof}

	\section{ Homological criteria for representation stability}
	
			Let $k$ be a field, and let $X$ be a Hausdorff, locally contractible toplogical space.  In this section, using the functoriality of the spectral sequence of Theorem \ref{FunctorialSS}, we obtain a spectral sequence of $\FI$ modules that converges to the cohomology of configuration space with coefficients in $k$.  Then we give a necessary and sufficient criterion for the $E_1$ page to be finitely generated in cohomological degrees $\leq k$.  This yields the criterion for the finite generation of cohomology, Theorem \ref{introcriterion}, since \cite{church2014fi}  shows that finite generation passes along spectral sequences.  Along the way, we show that the $E_1$ page is a free $\FI$ module and give generators for it. 
	
		We define how $\FI$ acts on the spectral sequences contstructed in Section 2, then describe the $E_1$ page as an $\FI$ module, and lastly prove the criterion for finite generation  as Theorem \ref{criterion}  Recall the definition of $\FI$ modules and partitions
		\begin{defn}
			The category $\FI$ has objects finite sets, and maps $\FI(a,b) = \{ \text{ injections } a \into b\}$.  For $k$ a commutative ring, an $\FI$ module is a functor from $\FI$ to the category of $k$ modules.  
		\end{defn}

		\begin{defn}[Conventions for partitions]
			Let ${\widehat P}([n])$, or just $\rP(n)$, denote the \emph{lattice of partitions} of the set $[n] = \{1, \dots, n\}$ ordered by refinement. Formally, the set of partitions of $n$ into $m$ blocks is $$\{ \text{ surjections  } e: [n] \onto [m] \}/S_m.$$
			
			 Given an injection $f: [n] \to [m]$ and a partition $p$ of $n$ we write $fp$ for the partition of $[m]$ such that $\im f$ is partitioned by $p$ under the identification with $n$ and  such that no $i,j$ complement of  $\im f$ are in the same block.  If $p  \leq p'$ then $fp \leq fp'$. 
			 
			  The construction $p \mapsto fp$ defines an action of $\FI$ on the partition posets, or more formally a functor $Q: \FI \to Poset$, such that $Q_n = \rP(n)$ and $Qf: \rP(n) \to \rP(m)$ takes $p$ to $fp$ .
			
			 A partition $q$ of $[m]$ is \emph{irreducible} if it is not  equal to $fp$ for any inclusion $f: n \to m$ for $n < m$, or equivalently if all the blocks of $q$ have more than one element. 
			
			 The lattice $\rP(n)$ is ranked: we define the \emph{rank of $p$} be $r(p) = n - \# \rm blocks$. We have that $r (p) = r(fp)$. In other words,  the action of $\FI$ preserves the rank. 
			 
			 For a partition $p$ of $[n]$ and $q$ of $[n]$ we have a partition $p \sqcup q$ of $[m] \sqcup [n]$. \QED
		\end{defn}
		
		Next we describe how $\FI$ acts on powers of $X$ and hence on the spectral sequences from the previous section.
		\begin{defn}[Action of $\FI$]
			 Let $X$ be a topological space, and let $\cF$ be a sheaf of $k$ modules on $X$ with a distinguished global section  $s: \ku \to \cF$. From this data, we describe a functor  $\FI \to \cal X$. 
			 
			   First note that the $n$-fold product of $X$, is the same as the topological space of maps $[n] \to X$.  For a partition $p$ of $n$ (given by $n \to m$),  write $Z_p$ for the collection of all maps that $n \to X$ that factor through $p$ and $U_p$ for $X^n - Z_p$.  Thus our convention is that $Z_{[n]} =  X^{[n]}$, and $U_{[n]} = \emptyset$.  We also write $\cF^{\boxtimes [n]}$ for the $n$th external tensor power of $\cF$, a sheaf on $X^{[n]}$.  
			   
			   The functor $\FI \to \cX$ takes $[n]$ to $(X^{[n]}, \{Z_p\}_{p \in \rP(n) - 1}, \rP(n), \cF^{\sqtimes n})$.   An injection, $g: [n]\into [m]$ gives a projection map $g: X^{[m]} \to X^{[n]}$, and the preimage of $Z_p$ under this map is $Z_{gp}$, where $g: \rP(n)  \to \rP(m)$ is as described above.  Using the natural isomorphism  $g^* \cF^{\sqtimes [n]} \simeq \cF^{\sqtimes \im g } \sqtimes k^{\sqtimes [m] - \im g}$,  we have $g$ act by  $$ {1^{\sqtimes \im g} \sqtimes s^{\sqtimes [m]-\im g} : ~~ } \cF^{ \im g} \sqtimes \ku^{\sqtimes [m] -\im g} \to \cF^{\sqtimes [m]}. $$    \QED
		   
		\end{defn}			   
		
					   Let $\Conf_n(X)$  be the \emph{ordered configuration space} of $n$ points in $X$. Then $X^{[n]} - \cup_{p \in \rP(n)} Z_p = \Conf_n(X)$

		\begin{ex}
			The partitions of $3$ are $$\{1\}\{2\}\{3\},~ \{1\}\{23\},~  \{2\}\{13\},~  \{3\}\{12\}, \text{ and } \{123\}$$ with $\{1\}\{2\}\{3\} \geq \{12\}\{3\} \geq \{123\}$.  If we take the inclusion  $f: [3] \to [5]$ with $1 \mapsto 1, 2 \mapsto 5, 3 \mapsto 3,$  then $f\{12\}\{3\} =\{15\}\{3\}\{2\}\{4\}$  The irreducible partitions of $[4]$ are $\{12\}\{34\}$ and $\{1234\}$
			
			For $X$ a topological space, we have $Z_{\{12\}\{34\}} = \{ (x_1, x_2, x_3, x_4) \in X^4 ~|~ x_1 = x_2 \text{ and } x_3 = x_4\}$
		\end{ex}

		From the functor $\FI \to {\cX}$, and the fact that the partition poset is graded Cohen Macaulay \cite{wachs2006poset}  ($\Tor_i(S_{\hat 1}, S_p)$  is only nonzero in degree $i =\depth p$), we obtain the following spectral sequence:
		\begin{prop}[Spectral sequence of $\FI$ modules] \label{FISS}
		 Let $k$ be a field. Then there is a spectral sequence of {\rm $\FI$} modules  converging to the cohomology of configuration space as an {\rm $\FI$} module  $E_1^{i, -j} =$
				$$  \bigoplus_{s \in \Z,~ p \in {\rP}(\bdot), ~s - r(p) = i - j , ~ 2 r( p) = i - s} H^{s}(X^\bdot,  U_p, \cF) \otimes \Tor^{{\widehat P}(\bdot)}_{t}(S_{\hat 1}, S_p)  \implies H^{i-j}(  X^\bdot - \cup_{p \in {\widehat P}(\bdot)-1} Z_p, \cF). $$
			In other words, for each $p \in \rP(n)$, $s \in \Z$, the term $H^{s}(X^n, U_p, \cF) \otimes \Tor_{r(p)}(S_{\hat 1}, S_p)$ appears exactly once on the $E_1$ page, in the place $E_1^{-i,j} = E_1^{2 r(p) + s, -3 r(p)}$.
		\end{prop}


		In order to identify the action on the $E_1$ page, we use the Kunneth formula to describe the following $\FI$ module: 
		\begin{prop}\label{identifyy}
			Let $k$ be a field, and $X$ be Hausdorff and locally contractible.  Let $\cF$ be a sheaf of $k$ vector spaces on $X$ such that $H^0(X, \cF) = {k}$ with distinguished section $x$.     For any $p_0 $ a partition of $n_0$ and $d \in \N$, the {\rm $\FI$} module $H^d(X,U_{p_0}, \cF^{\sqtimes [n]})$  defined by: 
			$$ n \mapsto\bigoplus_{\{g: [n_0] \into[ n]\}/\bS_{n_0}}  H^d \left(X^{[n]},U_{g p_0}, ~ \cF^{\boxtimes n} \right),$$ 
			is freely generated as follows:  for each $r \leq d$ and integer partition of $d -r = l_1  + \dots + l_m$, with  $l_1 \geq \dots \geq l_m \geq 1$, the  $\bS_{m + n_0}$ module
			  $$\bigoplus_{\{g: [m] \into [m] \sqcup [n_0]\} } H^r(X^{[m] \sqcup [n_0] - \im g} , U_{g p_0}, \cF^{\sqtimes [n_0]}) \otimes \bigotimes_{y \in \im g} H^{l_{g \inv y}} (X, \cF)  $$
			$$=  {\rm Ind}_{\bS_m \times \bS_n}^{\bS_{n+m}} \left( H^r(X^{n_0} , X^{n_0} - Z_\lambda, \cF^{\sqtimes n_0}) \otimes \bigoplus_{\sigma \in \bS_m}  H^{l_{\sigma1}}(X, \cF) \otimes \dots \otimes  H^{l_{\sigma m}}(X, \cF)   \right)$$  
			generates an {\rm $\FI$} submodule that is free (i.e. isomorphic to its induction to {\rm $\FI$}), and these submodules give a direct sum decomposition of $H^d(\rR i^!_{p_0} \cF^{\sqtimes [n]})$.
			
		\end{prop}
		\begin{proof}
			Note that $X^{[n]} - Z_{gp_0} = X^{[n]-{ \rm im }g} \times( X^{[n_0]} - Z_{p_0})$, so the pair $(X^{[n]}, X^{[n]} - Z_{gp_0}) = (X, X)^{[n] - \im g} \times (X^{[n_0]} , Z_{p_0})$.  Thus by the Kunneth formula, Proposition \ref{kunneth}, and since $k$ is a field,  we have that
			 $$ \bigoplus_{\{g :  \into [n]\}/\bS_{n_0}}  H^\bdot \left(X^{[n]}, U_{gp_0}, ~ \cF^{\boxtimes n} \right) = \bigoplus_{\{g : [n_0] \into [n]\}/S_{n_0}} H^\bdot(X, \cF)^{\otimes [n] - \im g} \otimes H^\bdot( X^{n_0}, U_p, \cF^{[n_0]})$$
			The piece in cohomological degree $d$ is 
			$$ \bigoplus_{\{g : [n_0] \into [n]\}/\bS_{n_0}}~ ~ \bigoplus_{c_t \in \N,~ t \in [n] - \im g,~ r \in \N, ~ r + \sum_t c_t = d}  H^r(X^{[n_0]}, U_{p_0}, \cF^{\sqtimes n_0}) \otimes \bigotimes_{t \in n - \im g}  H^{c_t}(X, \cF). $$
			An injection $h: n \into m$ acts by mapping the summand corresponding to $g, r$  and $c_t, t \in n - \im(g)$ to the summand corresponding to $h \circ g,  r$ and the sequence $c_s$ defined by: $c_s = c_t$ if $s \in {\im h} - \im( g \circ h) $ and $c_s = 0$ if $s \in [m] - \im h$.  Up to a sign, the maps uses isomorphism $x: k \to H^0(X, \cF)$ tensored with the action of $\bS_{n_0}$ on $ H^r(X^{[n_0]}, U_{p_0}, \cF^{\sqtimes [n_0]})$.  The $\bS_m$  module corresponding to  $d -r = l_1  + \dots + l_m$, is the $\bS_m$ submodule of the degree $[m] \sqcup [n_0]$ piece which is spanned by the summand corresponding to the canonical $g = [n_0] \into [n_0] \sqcup [m]$ and $c_i = l_i$ for $i \in [n_0] \sqcup [m] - [n_0] = [m]$.  We can generate every summand corresponding to the same integer partition padded by zeros.  And since the action of $\FI$ preserves the nonzero part of the partition, the submodules corresponding to these summands  intersect trivially.  Finally, recall that the free $\FI$ module on an $\bS_m$ module $M$ is $t \mapsto M \otimes_{\bS_m} k\{ m \into t\} = \bigoplus_{\{m \into t\}/\bS_t} M$, so that each submodule spanned is free. 
		\end{proof}

	It is essential for us that $\FI$ modules are Noetherian, which implies finite generation passes along spectral sequences. The following theorem appears in \cite{church2014fi}.
		\begin{thm}[CEFN] \label{noetherianity}
			Let $k$ be a Noetherian ring.  Then a submodule of a finitely generated {\rm $\FI$} module is finitely generated.
		\end{thm}
		\begin{cor}[CEFN]
			Suppose a first quadrant spectral sequence of {\rm $\FI$} modules $E_1^{p,-q}$, converging to an  {\rm $\FI$} module $H^{p+q}$.  If each $E_1^{p,-q}$ is a finitely generated $\FI$ module, then $H^{p-q}$ is finitely generated
		\end{cor}

		The following is not a direct consequence of  the above theorem 
		but is easier to prove:
		\begin{prop}[CEF] \label{fgTensor}
			If $M,N$ are finitely generated {\rm $\FI$} modules, then $M \otimes_{k} N$ is a finitely generated {\rm $\FI$} module.
		\end{prop}
		\begin{proof}
			 By right exactness, it suffices to show that the tensor product  ${k} \FI(n, -) \otimes {k} \FI(m, -)$ is finitely generated. This is the $\FI$ module: $$t \mapsto {k} \{ \text{pairs } (f : m \into t, g: n \into t) \} = \bigoplus_{a \leq {\rm min}(m,n)} \bigoplus_{\{\text{spans } m  \hookleftarrow a \hookrightarrow n \}/S_a} {k} \FI( m \sqcup_a n, t),$$ which is finitely generated
		\end{proof}
		
		Next we characterize when the degree $\leq k$ piece of the spectral sequence is a finitely generated $\FI$ module, by way of an auxiliary function.  
		\begin{defn}
				 For any partition $p$ of $n$, let  ${\rm van}(n) = \max \{t~|~ H^{t}(X^n, U_ p, \cF^{\sqtimes n}) = 0\}$.  In particular, for the indiscrete partition $[n]$ corresponding to the diagonal embedding we have ${\rm van}([n]) = \max \{t~|~ H^t(X^n, X^n - X, \cF^{\sqtimes n}) = 0\}$
		\end{defn}
		
		\begin{prop} \label{evilprop}
			Let $k$ be a field and let $X$ be locally contractible and Hausdorff.  Then the degree $\leq k$ piece of the $E_1$ page of the {\rm $\FI$} module local cohomology spectral sequence in Proposition \ref{FISS},  $\bigoplus_{i - j \leq k} E_1^{i, -j}$ is finitely generated  if and only if ${\rm van}([n]) - r([n]) \geq 0$ and for every $n$ and ${\rm van}([n]) - r([n]) \geq k$ for  $n \gg 0$.    
		\end{prop}
			\begin{proof}
		
			As a graded $\FI$ module, the degree $d$ piece of the  $E_1$ page admits a direct sum decomposition: $$ \bigoplus_{  n_0 \in \N, ~  \{ p \in {\widehat P}(n_0) \text{ irreducible} \} / S_{n_0}, ~s \in \N,~ s - r(p_0) = d   } \left(  \bigoplus_{\{f: n_0 \into n\}/S_{n_0}} H^s(X^{[n]} , U_{f\lambda}, \cF^{\sqtimes[n]}) \otimes_{k} \Tor_{ r(\lambda)}(S_{\hat 1}, S_{f \lambda}) \right).$$  
			We write $H^s(X^{n_0}, U_{p_0})  \otimes H^{\rP(n)}_{\lambda}$ for  the summand  corresponding to $s,p_0$. We claim that the degree $d$ piece is finitely generated if and only if there are only finitely many $n_0, p_0, s$  with $p_0$ irreducible  such that $s - r(p_0) = d$ and $H^s(X^{n_0}, U_{p_0}) \neq 0$.  First, if infinitely many are nonzero then we get an infinite direct sum of free modules, and tensoring each with $\Tor_{ r(p_0)}(S_{\hat 1}, S_{f p_0})$ gives another infinite sum of free modules by Proposition \ref{identifyy}, and so we get an infinite sum on the $E_1$ page in degree $d$, which is not finitely generated.  Conversely, we have that $ \bigoplus_{\{f: n_0 \into \bdot\}/S_n} \Tor_{r( \lambda)}(S_{\hat 1}, S_{f\lambda}) \subset H^{r( \lambda)} ( \Conf_\bdot \R^2 , \ku)$ is finitely generated.  So if only finitely many $n_0, p_0, s$ have $H^s(X^{n_0}, U_{p_0}) $ nonzero, then tensoring each with $\bigoplus_{\{f: n_0 \into m\}/S_n} \Tor_{r( \lambda)}(S_{\hat 1}, S_{f\lambda})$ preserves finite generation, by Proposition \ref{fgTensor}.
			
			Next, there are only finitely many $n_0, p_0, s$  with $p_0$ irreducible and such that $s - r(p_0) = d$ and $H^s(X^{n_0}, U_{p_0}) \neq 0$ iff there are only finitely many  finite  natural number sequences   $c_1 \geq \dots \geq c_m \geq 1$ and $b \in \N$, $p_0$ a partition of $n_0$ such that $ b - r(p_0)  + \sum c_i = k $ and $H^{c_1}(X, \cF) \otimes \dots \otimes H^{c_m}(X, \cF) \otimes H^b(X^{n_0}, U_{p_0}, \cF ) \neq 0$ .  This follows from the fact that $H^s(X^{n_0}, U_{p_0}) \neq 0$ iff  there is some choice of $c_1 \geq \dots \geq c_m \geq 1$ and $b$ as before, and for each $s, p_0, n_0$ there are at most finitely many choices of $c_i, b$ such that $b + \sum c_i - r(p_0) = k$.   
		
			 Thus we have that the $E_2$ page is finitely generated in cohomological degree $d$ for all $d \leq k$ iff there are only  finitely many $c_1 \geq \dots \geq c_m \geq 1 , b, p_0$ such that $b - r(p_0) + \sum_i c_i \leq k$ and $ H^{c_1}(X, \cF) \otimes \dots \otimes H^{c_m}(X, \cF) \otimes H^b(X^{[n_0]}, U_{p_0} \cF^{\sqtimes [n_0]}  \neq 0$.  This is true iff there are only finitely many $b, p_0$ such that $b - r(p_0) \leq k$ and $H^b(X^{[n_0]}, U_{p_0} \cF^{\sqtimes [n_0]}  \neq 0$, because each sequence $c_i, b, p_0$ gives such a $b, p_0$, each $b, p_0$ arises in this way, and for each $b, p_0$ there are only fintitely many positive sequences $c_i$ that satisfy $b - r(p_0) + \sum_i c_i \leq k$.  
			 
			 For each $p_0$, there is a $b$ such that $b - r(p_0) \leq k$ and $H^b(X^{[n_0]}, U_{p_0} ,\cF^{\sqtimes [n_0]})  \neq 0$   if and only if $\van(p_0) \leq k$, and for a fixed $p_0$ there at most finitely many $b \geq 0$ with $b \leq k + r(p_0)$. So the $E_1$ page is finitely generated in degree $\leq k$ iff there are finitely many irreducible partitions $p_0$ with $\van(p_0) - r(p_0) \leq k$.

			  Now  for partitions $p, q$ of $a, b$ we have that $U_{p \sqcup q} = X^{a} \times X^{b} - Z_p \times Z_q$ so  the Kunneth formula for local cohomology shows that $\van(p \sqcup q) =\van(p)  + \van(q)$.  Every irreducible parition $p_0$ is  the disjoint union of discrete partitions $[n]$ for $n \geq 2$.  So  if we have $\van([n])  - r([n]) \leq k$ for only finitely many $n$, and $\van([n])  - r([n]) \geq 1$ for $n \geq 2$, then there are only finitely many ways of combining the blocks $[n],~ n \geq 2$ to get an irreducible partition $p_0$ with $\van(p_0) - r(p_0) \leq k$.  Conversely if there is an $[n]$ with $\van([n])  - r([n]) \leq 0$, then $\sqcup_{i = 1}^M [n]$ gives an infinite sequence of irreducible partitions with $\van(p_0) - r(p_0) \leq k$, and if there are infinitely many $n$ with $\van([n]) - r([n]) \leq k$ then there are trivially infinitely many irreducible partitions $p_0$ with $\van(p_0) - r(p_0) \leq k$.
			    \end{proof}

			  Proposition \ref{evilprop} and \ref{noetherianity} together imply our main theorem about  configuration spaces, which gives a criterion for $H^{\leq c}$ to be a finitely generated $\FI$ module.  Here we write it in terms of the auxiliary function ${\rm van}(n)$.

		\begin{thm}[Criterion for representation stability]\label{criterion}
			Let $X$ be a Hausdorff, locally contractible topological space and $\cF$ a sheaf of $k$ vector spaces on $X$ such that $H^0(X, \cF) = {k}$ and $H^i(X, \cF)$ is finite dimensional for $i > 0$.    If 
			\begin{enumerate}
				\item $\van(n) \geq n -1$ for $n \geq 2$  and
				\item  ${\rm van}(n) \geq n + c$ for $n \gg 0$,
			\end{enumerate}
			 then $H^i(\Conf_\bdot(X), \cF)$ is a finitely generated {\rm $\FI$} module for all $i < c$.
		\end{thm}


			We give the proofs of corollaries from the introduction:  
		\begin{proof}[Proof of Corollary \ref{timesr2}]
			 Let $\Delta^X_n$ be the $n$th diagonal embedding of $X$. Then $H^i(X^n, X^n - \Delta^X_n X, k) = H^{i} (  \rR \pi_*\Delta^{X!}_n \ku)$.    If a complex of sheaves $C$ has homology (as a complex of sheaves) supported in cohomological degree $\geq i$, then so does $\rR^i\pi_* C$, so it suffices for $\rR \Delta^{X!}_n \ku$ to be supported in cohomological degree $\geq n$ when $n \geq 2$, and to have the range where it vanishes grow faster than $n + c$ for any constant $c$.  Since $\rR \Delta_n^! \ku$ is a complex of sheaves, this can be checked locally, and we have by the Kunneth formula, Proposition \ref{kunneth} and locality of $\rR i^!$, Proposition \ref{localization},  that $\rR\Delta^{X!}_n \ku |_{U_p} = \rR\Delta^{V_p !}_n \ku \sqtimes \rR\Delta^{\R^2 !}_ n\ku$.  But for $\R^2$ we have the diagonal embedding $\R^2 \into \R^{2n}$, so taking the cone from the pushforward of the complement $\R^{2n -2}$ we see that  $\rR\Delta^{\R^2 !}_ n\ku = \ku[-(2n - 2)]$ is concentrated in cohomological degree $\geq 2n -2$. So by the external tensor product, so is $\rR\Delta^{X!}_n \ku |_{U_p}$. 
  		\end{proof}
  		
		\begin{proof}[Proof of Corollary \ref{dualcorr2}]
			Proposition \ref{dualizing} shows that $H^i(X^n, X, \ku) = H^{i} (\rR \pi_* \rR \Hom(\omega_X^{\otimes n}, \omega_X)$, for $X$ locally contractible and embeddable into euclidean space. As in the proof of Corollary \ref{timesr2}, to satisify the hypotheses of Theorem \ref{introcriterion}, it suffices for $\rR \Hom(\omega_X^{\otimes n}, \omega_X)$ to be concentrated in cohomological degree $\geq 2n -2$, and thus for $\omega_X$ to be concentrated in homological degree $\geq 2$.
			
			The complex $\omega_X$ is concentrated in homological degree $\geq 2$ if and only if every stalk of $\omega_X$ has homological degree $\geq 2$.  For a contractible neighborhood $j: U \to X$ of $p$ we have that $i_p^* \pi_X^! k = i_p^* \pi_U^! k$, and for the inclusion $l: U - p \to U$ we have the triangle $$l_! l^! \pi_U^! k \to \pi_U^! k \to i_! i^* \pi_U^! k,$$  and applying $\pi_{U!}$  we get a triangle involving the homology of $U-p$, the homology of $U$ and $H_\bdot(i_p^* \omega_X)$.  So we see that $H_i(i_p^* \omega_X) = \tilde H_{i-1}(U- p, k)$.  Thus if $U-p$ is connected, the homology of the stalk vanishes in homological degrees $< 2$.  
		\end{proof}


 \begin{appendices}

 \section{Background on Local Cohomology}\label{loccohom}

			Let $X$ be a topological space,  $i: Z \to X$ be a closed subset and $j: U \to Z$ its open complement.  In this section, we give the background we need  on  $H^\bdot(X, U, \cF)$, called the local cohomology of $Z$ in $X$.       
			
			  Then the functor $i^!: \Sh(X) \to \Sh(Z)$ is  defined by $$i^!(\cF)(V \cap Z) = \{s \in \cF(V) ~|~ s \text{  is supported on }  Z\}$$ It is left exact and we may write $\rR i^!$ for its right derived functors.  However,  we will use an alternate definition of $\rR i^!$, from which the properties that we need follow easily.  Recall that there is a natural map $\cF \to \rR j_*j^* \cF$, which is modelled by $\cJ^{\bdot} \to j_* j^*\cJ^\bdot  $
			\begin{defn}
				We define $\rR i^! \cF  := \cone( \cF \to  \rR j_*j^*\cF)[-1]$
			\end{defn}
			
			If we take the cohomology of the global sections of $\rR i^! \cF$ we get the relative sheaf cohomology of $U,X$:  $$H^i(\rR\pi_*\rR i^! \cF) = H^{i}(X, U, \cF) $$
			
			We use the following version of the Kunneth formula to make reductions above:
			\begin{prop}[Kunneth for local cohomology]\label{kunneth}
				Let $X, X'$ be locally contractible  topological spaces, let $\cF, \cF'$ sheaves of $k$ modules , and let $Z, Z'$ be closed subsets with inclusions $i,i'$ and complements $U,U'$.  On  $T \times T'$ we have the closed subset $Z \times Z'$.  Then $R(i \times {i'})^! (\cF \boxtimes \cF ') \simeq (\rR i^!\cF  )\boxtimes (\rR{i'}^! \cF'$.  In particular for $k$ a field, then we have that $H^\bdot(X \times X',  X \times X' - Z \times Z', \cF \sqtimes \cF') = H^\bdot(X, U, \cF) \otimes H^\bdot(X', U', \cF')$   

			\end{prop}
			\begin{proof}
				We have that $X \times X' - Z \times Z' = (U \times X') \cup (X \times U')$.  So by Mayer-Vietoris and the Kunneth formula\footnote{ Here we use that $X, X'$ are locally contractible to apply Kunneth.}, we have an exact sequence: $$ \rR j_* ( \cF \boxtimes \cF') |_{X \times X' - Z \times Z'}  \to (\rR j_* \cF|_U) \boxtimes \cF'  \oplus  \cF \boxtimes (\rR j_* \cF'|_{U'})  \to   (\rR j_* \cF|_U)\boxtimes  (\rR j_* \cF'|_{U'}), $$ so that $$ R(i \times {i'})^! (\cF \boxtimes \cF )' [ 1] \simeq \text{ {Tot\rm} }  \big ( \cF \boxtimes \cF' \to (\rR j_* \cF|_U) \boxtimes \cF'  \oplus  \cF \boxtimes (\rR j_* \cF'|_{U'})  \to   (\rR j'_* \cF|_U)\boxtimes  (\rR j_* \cF'|_{U'}) \big )$$  $$= \cone(  \cF \to   \rR j_*  \cF|_U)  \boxtimes \cone(  \cF \to   \rR j'_*  \cF'|_{U'})   =  (\rR i^!\cF [1] )\boxtimes (\rR{i'}^! \cF'[1])$$
			\end{proof}
			It is also important, that $\rR i^!$ is determined locally.  
			\begin{prop}[Localization]\label{localization}
				Let $V \subset X$ be an open subset and write $i|_V: Z \cap V \to V$ for the inclusion.  Then $ (\rR i^! \cF)|_V   = \rR i^!|_V (\cF|_V)  $ 
			\end{prop}
			
			Lastly, to restate the vanishing criterion,  we  use a Kunneth formula  for the dualizing complex of a space:
			\begin{prop}[Kunneth for Dualizing Complexes]\label{generalkunneth}
				Let $i_X: X \to \R^n$ and $i_Y : Y \to \R^m$ be two locally contractible topological spaces together with closed embeddings into euclidean space,  $j_X, j_Y$ the embeddings of their open complements. Let $V, W$ be free $k$ modules.   We write $\pi_X : X \to *$ and $\pi_Y: Y \to *$ for the respective terminal maps.  Then  $\pi_X^! V \simeq Ri_X^!  \underline V [n]$ and  $\pi_Y^! W \simeq Ri_Y^! W [m]$ are isomorphic complexes in the derived category.  And we have that $(\pi_X \times \pi_Y)^! (V \otimes W) \simeq \pi_X^! V \sqtimes \pi_Y^! W$

			\end{prop}
			\begin{proof}
				By \cite{kashiwara2013sheaves}, Chapter 3, for any closed embedding $Z \into \R^l$,  and vector space $U$, we have $\rR\pi_X^! U \simeq \rR i_Z^! \pi_{\R^m}^! U \simeq \rR i_Z^! U[l]$.  So for the closed embedding $Z = X \times Y \into \R^n \times \R^m$ we have that $\pi^!_{X\times Y} V \otimes W = \rR(i_X^ \times i_Y)^! V[n] \otimes W[m] = \rR i_X^! V[n] \sqtimes \rR i_Y^! W[m] $ by \ref{kunneth}.
			\end{proof}
			
			\begin{prop}\label{dualizing}
				Let $X$ be a locally contractible topological space, embeddable in $\R^l$,  with diagonal embedding $\Delta_n: X \to X^n$.  Then $\rR \Delta_n^! \ku ={\omega_X^*}^{\otimes n-1}$
			\end{prop}
			\begin{proof}
				    Proposition \ref{generalkunneth} states that $\omega_{X^n} = \omega_X^{\sqtimes n}$.  Verdier duality intertwines $f^!$ and $f^*$  \cite{kashiwara2013sheaves}, so we have 
	      $\Delta^! \ku^{\sqtimes n} = \rR \underline{ Hom}(  ~\rR \Delta^* ~ \rR \underline{Hom}(\ku^{\sqtimes n}, \omega_{X^n}) , \omega_X)$.   Therefore we have $\rR\Delta_n^! \ku =  \underline{Hom}(\omega_X^{\otimes n} , \omega_X) = ({\omega_X^*})^{\otimes n-1}$.
	     		\end{proof}

 \end{appendices}

				\printbibliography

\end{document}